\documentclass[10pt, a4paper]{article}

\usepackage{amsmath}
\usepackage{amssymb}
\usepackage{color}
\usepackage{graphicx}
\usepackage{amsthm}
\newtheorem{thm}{Theorem}[section]

\newtheorem{lem}[thm]{Lemma}

\newtheorem{rem}{Remark}[section]
\newtheorem{assump}{Assumption}

\numberwithin{equation}{section}
\numberwithin{thm}{section}
\numberwithin{rem}{section}

\title{Convergence Rates and Explicit Error Bounds of 
Hill's Method for Spectra of Self-Adjoint Differential Operators}

\author{
Ken'ichiro Tanaka$^{\ast}$ $^{\dagger}$, 
Sunao Murashige$^{\ast}$ $^{\ddagger}$\\
$^{\ast}$School of Systems Information Science, Future University Hakodate, \\
116-2 Kamedanakano-cho, Hakodate, Hokkaido, 041-8655, Japan\\
$^{\dagger}$ ketanaka@fun.ac.jp,\ $^{\ddagger}$ murasige@fun.ac.jp
}

\date{May 15, 2012}

\begin{document}

\maketitle

\begin{abstract}
We present the convergence rates and the explicit error bounds 
of Hill's method, which is a numerical method for computing the spectra of 
ordinary differential operators with periodic coefficients. 
This method approximates the operator by a finite dimensional matrix. 
On the assumption that the operator is self-adjoint,
it is shown that, under some conditions, we can obtain the convergence rates of eigenvalues with respect to
the dimension and the explicit error bounds.
Numerical examples demonstrate that we can verify these conditions 
using Gershgorin's theorem for some real problems.
Main theorems are proved using the Dunford integrals which project an eigenvector to the corresponding eigenspace.\\


\end{abstract}

\section{Introduction}

This paper considers Hill's method \cite{bib:Hill_HillsMethod}, which is 
a numerical method for computing spectra of ordinary differential operators $S_{p}$ with periodic coefficients:
\begin{align}
S_{p} = 
\frac{\mathrm{d}^p}{\mathrm{d}x^p} + \sum_{j=0}^{p-1} \tilde{f}_{j}(x) \frac{\mathrm{d}^{j}}{\mathrm{d}x^{j}},
\label{eq:main_prob}
\end{align}
where $x \in \mathbf{R}$, and
the functions $\tilde{f}_{j}\ (j=0,1,\ldots, p-1)$ are $C^{\infty}$ and satisfy 
$\tilde{f}_{j}(x+L) = \tilde{f}_{j}(x)$ for some real positive constant $L$.
The eigenvalue problem for $S_{p}$ is described as $S_p \phi = \lambda \phi$, 
where $\lambda$ is an eigenvalue and $\phi$ is an eigenvector. 
This eigenvalue problem often appears in physical problems such as linear stability analysis of 
periodic solutions of nonlinear wave equations \cite{bib:DeconinckKutz2006}.
The set of eigenvalues $\lambda$ is included in the spectrum $\sigma(S_p)$. 
Accordingly, computation of $\sigma(S_p)$ is important in both theoretical and practical points of view.
It was reported in \cite{bib:DeconinckKutz2006}
that Hill's method can produce very good computed results of $\sigma(S_p)$ for some problems. 
Note that implementation of Hill's method is straightforward. 

The ideas of 
Hill's method are to employ
the Floquet-Bloch decomposition of the spectrum $\sigma(S_{p})$ and 
to approximate the eigenvector $\phi$ by a finite Fourier series.
This Fourier series approximation generates an eigenvalue problem of finite dimension $D$ corresponding to the operator $S_{p}$.
The aims of this paper are to 
show the convergence rates of the approximate eigenvalues with respect to the dimension $D$, 
and to explicitly obtain the error bounds of them.

On the convergence property of Hill's method, 
Curtis and Deconinck \cite{bib:CurtisDeconinck2010}
proved that the exact eigenvalues exist near the computed ones for general cases of $S_{p}$, 
and showed that all exact eigenvalues can be approximated by the computed ones for the case of self-adjoint $S_{p}$.
Also, they obtained the convergence rate
for the case of self-adjoint $S_{p}$ and 
the constant coefficients $\tilde{f}_{1}, \ldots, \tilde{f}_{p-1}$. 
Johnson and Zumbrun \cite{bib:JohnsonZumbrun2012} 
investigated Hill's method
using the Evans function, of which the roots correspond to eigenvalues, 
for general cases of $S_{p}$. 
They showed that 
the approximate eigenvalues 
converge to the exact ones, but did not get the convergence rate of them.
Vainikko \cite{bib:Krasnoselskii_etal_1972}\cite{bib:Vainikko1967}
examined an approximation method for the eigenvalue problem $S_p \phi = \lambda \phi$
in an abstract framework, and
obtained the convergence rate, 
which is based on the resolvent norm convergence of approximate operators, 
for general cases of $S_{p}$. 
We can apply Vainikko's results to Hill's method. 
But it is difficult to directly evaluate the value of the convergence rate.
It should be noted that, 
although all of these convergence rates yield some error bounds of approximate eigenvalues
with unknown coefficients, any explicit error bounds have not been shown. 

In this paper, we 
give a priori estimates of the convergence rate and a posteriori explicit error bounds of Hill's method
for self-adjoint $S_{p}$ with the two cases of coefficient functions $\tilde{f}_{j}$ in  \eqref{eq:main_prob}, 
namely $\tilde{f}_{j}\in C^{\infty}$ and $\tilde{f}_{j}$ being analytic 
on some strip region containing the real line. 
These classes of $\tilde{f}_{j}$ are more general than \cite{bib:CurtisDeconinck2010}. 
The key ideas of these estimations are to project eigenvectors 
using the Dunford integrals \eqref{eq:DunfordIntP} and \eqref{eq:DunfordIntP_pl}, 
and to specify the disks around the exact eigenvalues using 
Gershgorin's theorem (Theorem~\ref{thm:Gershgorin}).

This paper is organized as follows. 
Section~\ref{sec:HillsMethod} describes Hill's method. 
In Section~\ref{sec:DiffOpHill}, 
we give the convergence rates and the explicit error bounds of Hill's method
for self-adjoint differential operators. 
Section~\ref{sec:NumComp} presents
numerical examples which supports our results. 
Section~\ref{sec:GenConv} summarizes the proofs of theorems in
Section~\ref{sec:DiffOpHill}.
Section~\ref{sec:ConcRem} concludes this paper.

\section{Hill's method}
\label{sec:HillsMethod}

This section describes Hill's method \cite{bib:Hill_HillsMethod} which is a numerical method for 
computing the spectum  $\sigma(S_p)$ of the ordinary differential operator $S_{p}$ with periodic coefficients 
$\tilde{f}_{j}$ 
defined by \eqref{eq:main_prob}.
This operator $S_p$ can be regarded as an operator $S_p: H^{p}(\mathbf{R}) \to L_2(\mathbf{R})$, 
where $L_2(\mathbf{R})$ is the Lebesgue space of square integrable functions on $\mathbf{R}$ and 
$H^{p}(\mathbf{R}) \subset L_2(\mathbf{R})$ is the Sobolev space of functions whose derivatives 
up to $p$-th order are square integrable. 
Hill's method approximates the elements of $\sigma(S_p)$
by the following two steps.\\

\noindent
{\bf Step 1: Floquet-Bloch decomposition}. 
In order to apply the Floquet theory, 
we introduce a new operator 
$S_p^{\mu}: H^{p}([0,L])_{\mathrm{per}} \to L_{2}([0, L])_{\mathrm{per}}$ defined by 
\begin{align}
S_p^{\mu} = \mathrm{e}^{-\mathrm{i} \mu x} S_p \mathrm{e}^{\mathrm{i} \mu x}, 
\label{eq:Def_S_p_mu}
\end{align}
where $\mu \in [0,\, 2\pi / L)$, and 
$L_{2}([0, L])_{\mathrm{per}}$ and $H^{p}([0,L])_{\mathrm{per}}$ are 
the Lebesgue space and the Sobolev space of periodic functions on $[0, L]$, 
respectively. 
Note that $S_p$ is defined on $\mathbf{R}$, whereas 
$S_p^{\mu}$ is on $[0, L]$. 
If  $S_p$ is self-adjoint, then $S_p^{\mu}$ is also self-adjoint. 
Moreover,  $S_p^{\mu}$ is explicitly written as 
\begin{align}
S_p^{\mu} = 
\frac{\mathrm{d}^{p}}{\mathrm{d}x^{p}} + 
\sum_{j=0}^{p-1} f_{j}(x) \frac{\mathrm{d}^{j}}{\mathrm{d}x^{j}}
\label{eq:Expr_S_p_mu}
\end{align}
for some periodic functions $f_{0},\ldots, f_{p-1}$ with period $L$.

It is known that the spectrum $\sigma(S_p^{\mu})$
consists of only
the eigenvalues of $S_p^{\mu}$, 
and the Floquet theory yields
\begin{align}
\sigma(S_p) = \bigcup_{\mu \in [0,\, 2\pi / L)} \sigma(S_p^{\mu}).
\label{eq:FloquetBloch}
\end{align}
This decomposition \eqref{eq:FloquetBloch} is called the Floquet-Bloch decomposition. 
Accordingly, it suffices to consider the eigenvalue problem 
$S_p^{\mu} \phi = \lambda \phi$ for $\mu \in [0,\, 2\pi / L)$, 
where $\lambda\in\mathbf{C}$ is an eigenvalue and $\phi\in L_2([0,L])_{\mathrm{per}}$ is an eigenvector. 
In addition, it should be noted that 
$\sigma(S_{p}^{\mu})$ is a discrete set of the eigenvalues of $S_{p}^{\mu}$ without accumulation points. 
This discreteness follows from the compactness of the resolvent of $S_{p}^{\mu}$
provided $\rho(S_{p}^{\mu}) \neq \emptyset$, where $\rho(S_{p}^{\mu}) = \mathbf{C} \setminus \sigma(S_{p}^{\mu})$.
See e.g.~Lemma~2 in \cite[Chapter XIX Section~2]{bib:DunfordSchwartzIII}, 
Lemma~3 in \cite[Chapter XIX Section~3]{bib:DunfordSchwartzIII} and
Lemma~16 in \cite[Chapter XIII Section~2]{bib:DunfordSchwartzII}.\\

\noindent
{\bf Step 2: Fourier series approximation}.
Since $\phi \in L_2([0,L])_{\mathrm{per}}$ has the Fourier series expansion:
\begin{align}
\phi(x) = \frac{1}{\sqrt{L}} \sum_{n=-\infty}^{\infty}\hat{\phi}_n\,  \exp \left( - \mathrm{i} \dfrac{2\pi n x}{L}\right),
\end{align}
$\phi $ can be approximated by the truncation of this series:
\begin{align}
(\hat{P}_N \phi)(x) = \frac{1}{\sqrt{L}} \sum_{n=-N}^{N}\hat{\phi}_n\, \exp \left( - \mathrm{i} \dfrac{2\pi n x}{L}\right).
\label{eq:Def_P_N}
\end{align}
This truncation reduces the eigenvalue problem 
$S_p^{\mu} \phi = \lambda \phi$ to a finite dimensional problem. 
More precisely, 
the problem
\begin{align}
S_{p, N}^{\mu} \phi_N = \lambda_N \phi_N 
\quad \text{with}\quad 
S_{p,N}^{\mu} = \hat{P}_N S_{p}^{\mu} \hat{P}_N
\label{eq:FiniteDimProblem}
\end{align}
gives approximate eigenvalues $\lambda_{N}$'s for the original problem $S_p^{\mu} \phi = \lambda \phi$. 

Since the problem \eqref{eq:FiniteDimProblem} is equivalent to a matrix eigenvalue problem, 
we can obtain the approximate values of the eigenvalues of $S_p^{\mu}$ using some standard numerical method.  
In the following sections, 
let $\sigma(S_{p,N}^{\mu})$ denote the set of eigenvalues of $S_{p,N}^{\mu}$.

\section{Convergence Rates and Error Bounds of Hill's Method}
\label{sec:DiffOpHill}

In this section, 
we present theorems about a priori estimates of the convergence rate and 
a posteriori explicit error bounds of Hill's method for the eigenvalue problem $S_p^{\mu} \phi = \lambda \phi$
with self-adjoint  $S_p^{\mu}$ 
on some assumptions. 
Their proofs are given in Section~\ref{sec:GenConv}. 
In what follows, we use the notations defined in Section~\ref{sec:HillsMethod}. 

\subsection{Assumptions}

First, we assume the self-adjointness of the operator $S_{p}^{\mu}$. 

\renewcommand{\theassump}{1}
\begin{assump}\label{assump:self-adjoint}
The operator $S_{p}^{\mu}$ is self-adjoint, i.e. $(S_{p}^{\mu})^{\ast} = S_{p}^{\mu}$, 
where $^{\ast}$ represents the adjoint of operators.
\end{assump}
\noindent
On Assumption~\ref{assump:self-adjoint}, 
$S_{p, N}^{\mu}$ in \eqref{eq:FiniteDimProblem} is also self-adjoint.

Next, 
we specify the smoothness of the coefficient functions $f_{j}$'s of 
$S_{p}^{\mu}$ in \eqref{eq:Expr_S_p_mu}.

\renewcommand{\theassump}{2a}
\begin{assump}
\label{assump:C_infty}
For the operator $S_{p}^{\mu}$, the coefficients $f_{0}, f_{1},\ldots, f_{p-1}$ are $C^{\infty}$.
\end{assump}
\noindent
In addition, 
as a special case of Assumption~\ref{assump:C_infty}, 
the following assumption is prepared.

\renewcommand{\theassump}{2b}
\begin{assump}
\label{assump:analytic}
For the operator $S_{p}^{\mu}$, the coefficients $f_{0}, f_{1},\ldots, f_{p-1}$ are analytic on 
a complex domain 
\begin{align}
\mathcal{D}_d = \{ z \in \mathbf{C}\ |\ |\mathrm{Im}\, z| < d\}
\label{eq:DefD_d}
\end{align}
for some $d>0$.
\end{assump}
\noindent
Assumption~\ref{assump:analytic} is often satisfied in real problems.

If $S_{p, N}^{\mu}$ approximates $S_{p}^{\mu}$ very well, 
we may expect that 
approximate eigenvalues $\lambda_{N}$'s for large enough $N$ 
should be included in the neighborhood $B_{\lambda}(r_{\lambda})$
of the corresponding exact eigenvalue $\lambda$, 
where 
\begin{align}
B_{\zeta}(r) = \{ z \in \mathbf{C} \mid | z - \zeta | \leq r \}.
\label{eq:DefDisk}
\end{align}
Then we prepare the following two assumptions. 

\renewcommand{\theassump}{3a}
\begin{assump}\label{assump:approx_cond_single}
For $\lambda \in \sigma(S_{p}^{\mu})$, there exists a positive real number $r_{\lambda}$
such that the following holds true: 
for some sequence $\{ \lambda_N \mid \lambda_N \in \sigma(S_{p, N}^{\mu}) \}$ 
and some positive integer $N_{0}$, 
if $N > N_{0}$ then
\begin{align}
 |\lambda - \lambda_N| \leq \frac{r_{\lambda}}{2}
\label{eq:3a_in}
\end{align}
and
\begin{align}
2r_{\lambda} \leq
\min 
\left\{ 
\mathrm{dist}(\lambda, \sigma(S_{p}^{\mu})\setminus \{\lambda \}),\, 
\mathrm{dist}(\lambda, \sigma(S_{p, N}^{\mu})\setminus \{\lambda_N \}) 
\right\}.
\label{eq:3a_out}
\end{align}
\end{assump}

Figure~\ref{fig:eigenvalues01} illustrates these conditions \eqref{eq:3a_in} and~\eqref{eq:3a_out}. 
When a single sequence $\{ \lambda_{N} \}$ for some $\lambda \in \sigma(S_{p}^{\mu})$ is detected, 
Assumption~\ref{assump:approx_cond_single} is effective for estimation of the convergence rate 
and the error bound of $\{ \lambda_{N} \}$. 
In general, 
there exist more than one sequence of approximate eigenvalues
for some $\lambda \in \sigma(S_{p}^{\mu})$. 
Then we set another assumption as follows:

\renewcommand{\theassump}{3b}
\begin{assump}\label{assump:approx_cond_plural}
Let $k$ be an integer with $k\geq 2$.
For $\lambda \in \sigma(S_{p}^{\mu})$, there exists a positive real number $r_{\lambda}$
such that the following holds true: 
for some sequences $\{ \lambda_{N, i} \mid \lambda_{N, i} \in \sigma(S_{p, N}^{\mu}) \}\ (i=1,2,\ldots k)$  
and some positive integer $N_{0}$, 
if $N > N_{0}$ then
\begin{align}
 |\lambda - \lambda_{N, i}| \leq \frac{r_{\lambda}}{2}
 \quad (i = 1,2,\ldots, k)
\label{eq:3b_in}
\end{align}
and 
\begin{align}
2r_{\lambda} \leq
\min 
\left\{
\mathrm{dist}(\lambda, \sigma(S_{p}^{\mu})\setminus \{\lambda \}),\, 
\mathrm{dist} \left( \lambda,\ \sigma(S_{p, N}^{\mu})\setminus \left( \bigcup_{i=1}^{k} \{\lambda_{N, i} \}\right) \right) 
\right\}.
\label{eq:3b_out}
\end{align}
\end{assump}

\begin{rem}
\label{rem:discrete_spectra}
Assumptions~\ref{assump:approx_cond_single} and~\ref{assump:approx_cond_plural} mean that, 
for sufficiently large $N$, 
the number of eigenvalues of $S_{p, N}^{\mu}$ approximating $\lambda \in \sigma(S_{p}^{\mu})$ is at most finite, 
and the other elements of $\sigma(S_{p}^{\mu})$ and $\sigma(S_{p, N}^{\mu})$ are relatively far from $\lambda$.
The finiteness of the approximate eigenvalues is based on 
the fact that the eigenspace of $S_{p}^{\mu}$ for any $\lambda$ is finite dimensional. 
\end{rem}

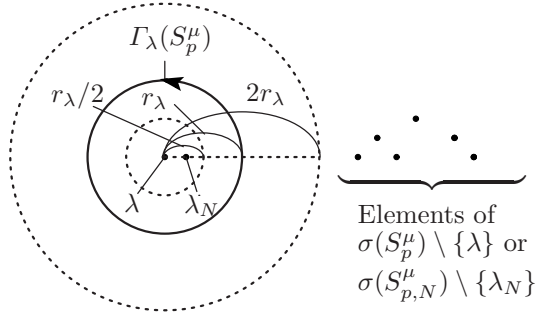
\begin{figure}[t]
\begin{center}
\unitlength 0.1in
\begin{picture}( 29.1500, 16.0000)( -3.1500,-18.0000)
%
{\color[named]{Black}{%
\special{pn 13}%
\special{ar 1000 1000 200 200  0.0000000 0.0600000}%
\special{ar 1000 1000 200 200  0.2400000 0.3000000}%
\special{ar 1000 1000 200 200  0.4800000 0.5400000}%
\special{ar 1000 1000 200 200  0.7200000 0.7800000}%
\special{ar 1000 1000 200 200  0.9600000 1.0200000}%
\special{ar 1000 1000 200 200  1.2000000 1.2600000}%
\special{ar 1000 1000 200 200  1.4400000 1.5000000}%
\special{ar 1000 1000 200 200  1.6800000 1.7400000}%
\special{ar 1000 1000 200 200  1.9200000 1.9800000}%
\special{ar 1000 1000 200 200  2.1600000 2.2200000}%
\special{ar 1000 1000 200 200  2.4000000 2.4600000}%
\special{ar 1000 1000 200 200  2.6400000 2.7000000}%
\special{ar 1000 1000 200 200  2.8800000 2.9400000}%
\special{ar 1000 1000 200 200  3.1200000 3.1800000}%
\special{ar 1000 1000 200 200  3.3600000 3.4200000}%
\special{ar 1000 1000 200 200  3.6000000 3.6600000}%
\special{ar 1000 1000 200 200  3.8400000 3.9000000}%
\special{ar 1000 1000 200 200  4.0800000 4.1400000}%
\special{ar 1000 1000 200 200  4.3200000 4.3800000}%
\special{ar 1000 1000 200 200  4.5600000 4.6200000}%
\special{ar 1000 1000 200 200  4.8000000 4.8600000}%
\special{ar 1000 1000 200 200  5.0400000 5.1000000}%
\special{ar 1000 1000 200 200  5.2800000 5.3400000}%
\special{ar 1000 1000 200 200  5.5200000 5.5800000}%
\special{ar 1000 1000 200 200  5.7600000 5.8200000}%
\special{ar 1000 1000 200 200  6.0000000 6.0600000}%
\special{ar 1000 1000 200 200  6.2400000 6.2832853}%
}}%
%
{\color[named]{Black}{%
\special{pn 13}%
\special{ar 1000 1000 400 400  0.0000000 6.2831853}%
}}%
\put(8.7000,-12.8500){\makebox(0,0)[rb]{$\lambda$}}%
%
{\color[named]{Black}{%
\special{pn 4}%
\special{sh 1}%
\special{ar 1000 1000 14 14 0  6.28318530717959E+0000}%
\special{sh 1}%
\special{ar 1110 1000 14 14 0  6.28318530717959E+0000}%
\special{sh 1}%
\special{ar 2000 1000 14 14 0  6.28318530717959E+0000}%
\special{sh 1}%
\special{ar 2200 1000 14 14 0  6.28318530717959E+0000}%
\special{sh 1}%
\special{ar 2600 1000 14 14 0  6.28318530717959E+0000}%
\special{sh 1}%
\special{ar 2600 1000 14 14 0  6.28318530717959E+0000}%
}}%
%
{\color[named]{Black}{%
\special{pn 13}%
\special{ar 1000 1000 800 800  0.0000000 0.0150000}%
\special{ar 1000 1000 800 800  0.0600000 0.0750000}%
\special{ar 1000 1000 800 800  0.1200000 0.1350000}%
\special{ar 1000 1000 800 800  0.1800000 0.1950000}%
\special{ar 1000 1000 800 800  0.2400000 0.2550000}%
\special{ar 1000 1000 800 800  0.3000000 0.3150000}%
\special{ar 1000 1000 800 800  0.3600000 0.3750000}%
\special{ar 1000 1000 800 800  0.4200000 0.4350000}%
\special{ar 1000 1000 800 800  0.4800000 0.4950000}%
\special{ar 1000 1000 800 800  0.5400000 0.5550000}%
\special{ar 1000 1000 800 800  0.6000000 0.6150000}%
\special{ar 1000 1000 800 800  0.6600000 0.6750000}%
\special{ar 1000 1000 800 800  0.7200000 0.7350000}%
\special{ar 1000 1000 800 800  0.7800000 0.7950000}%
\special{ar 1000 1000 800 800  0.8400000 0.8550000}%
\special{ar 1000 1000 800 800  0.9000000 0.9150000}%
\special{ar 1000 1000 800 800  0.9600000 0.9750000}%
\special{ar 1000 1000 800 800  1.0200000 1.0350000}%
\special{ar 1000 1000 800 800  1.0800000 1.0950000}%
\special{ar 1000 1000 800 800  1.1400000 1.1550000}%
\special{ar 1000 1000 800 800  1.2000000 1.2150000}%
\special{ar 1000 1000 800 800  1.2600000 1.2750000}%
\special{ar 1000 1000 800 800  1.3200000 1.3350000}%
\special{ar 1000 1000 800 800  1.3800000 1.3950000}%
\special{ar 1000 1000 800 800  1.4400000 1.4550000}%
\special{ar 1000 1000 800 800  1.5000000 1.5150000}%
\special{ar 1000 1000 800 800  1.5600000 1.5750000}%
\special{ar 1000 1000 800 800  1.6200000 1.6350000}%
\special{ar 1000 1000 800 800  1.6800000 1.6950000}%
\special{ar 1000 1000 800 800  1.7400000 1.7550000}%
\special{ar 1000 1000 800 800  1.8000000 1.8150000}%
\special{ar 1000 1000 800 800  1.8600000 1.8750000}%
\special{ar 1000 1000 800 800  1.9200000 1.9350000}%
\special{ar 1000 1000 800 800  1.9800000 1.9950000}%
\special{ar 1000 1000 800 800  2.0400000 2.0550000}%
\special{ar 1000 1000 800 800  2.1000000 2.1150000}%
\special{ar 1000 1000 800 800  2.1600000 2.1750000}%
\special{ar 1000 1000 800 800  2.2200000 2.2350000}%
\special{ar 1000 1000 800 800  2.2800000 2.2950000}%
\special{ar 1000 1000 800 800  2.3400000 2.3550000}%
\special{ar 1000 1000 800 800  2.4000000 2.4150000}%
\special{ar 1000 1000 800 800  2.4600000 2.4750000}%
\special{ar 1000 1000 800 800  2.5200000 2.5350000}%
\special{ar 1000 1000 800 800  2.5800000 2.5950000}%
\special{ar 1000 1000 800 800  2.6400000 2.6550000}%
\special{ar 1000 1000 800 800  2.7000000 2.7150000}%
\special{ar 1000 1000 800 800  2.7600000 2.7750000}%
\special{ar 1000 1000 800 800  2.8200000 2.8350000}%
\special{ar 1000 1000 800 800  2.8800000 2.8950000}%
\special{ar 1000 1000 800 800  2.9400000 2.9550000}%
\special{ar 1000 1000 800 800  3.0000000 3.0150000}%
\special{ar 1000 1000 800 800  3.0600000 3.0750000}%
\special{ar 1000 1000 800 800  3.1200000 3.1350000}%
\special{ar 1000 1000 800 800  3.1800000 3.1950000}%
\special{ar 1000 1000 800 800  3.2400000 3.2550000}%
\special{ar 1000 1000 800 800  3.3000000 3.3150000}%
\special{ar 1000 1000 800 800  3.3600000 3.3750000}%
\special{ar 1000 1000 800 800  3.4200000 3.4350000}%
\special{ar 1000 1000 800 800  3.4800000 3.4950000}%
\special{ar 1000 1000 800 800  3.5400000 3.5550000}%
\special{ar 1000 1000 800 800  3.6000000 3.6150000}%
\special{ar 1000 1000 800 800  3.6600000 3.6750000}%
\special{ar 1000 1000 800 800  3.7200000 3.7350000}%
\special{ar 1000 1000 800 800  3.7800000 3.7950000}%
\special{ar 1000 1000 800 800  3.8400000 3.8550000}%
\special{ar 1000 1000 800 800  3.9000000 3.9150000}%
\special{ar 1000 1000 800 800  3.9600000 3.9750000}%
\special{ar 1000 1000 800 800  4.0200000 4.0350000}%
\special{ar 1000 1000 800 800  4.0800000 4.0950000}%
\special{ar 1000 1000 800 800  4.1400000 4.1550000}%
\special{ar 1000 1000 800 800  4.2000000 4.2150000}%
\special{ar 1000 1000 800 800  4.2600000 4.2750000}%
\special{ar 1000 1000 800 800  4.3200000 4.3350000}%
\special{ar 1000 1000 800 800  4.3800000 4.3950000}%
\special{ar 1000 1000 800 800  4.4400000 4.4550000}%
\special{ar 1000 1000 800 800  4.5000000 4.5150000}%
\special{ar 1000 1000 800 800  4.5600000 4.5750000}%
\special{ar 1000 1000 800 800  4.6200000 4.6350000}%
\special{ar 1000 1000 800 800  4.6800000 4.6950000}%
\special{ar 1000 1000 800 800  4.7400000 4.7550000}%
\special{ar 1000 1000 800 800  4.8000000 4.8150000}%
\special{ar 1000 1000 800 800  4.8600000 4.8750000}%
\special{ar 1000 1000 800 800  4.9200000 4.9350000}%
\special{ar 1000 1000 800 800  4.9800000 4.9950000}%
\special{ar 1000 1000 800 800  5.0400000 5.0550000}%
\special{ar 1000 1000 800 800  5.1000000 5.1150000}%
\special{ar 1000 1000 800 800  5.1600000 5.1750000}%
\special{ar 1000 1000 800 800  5.2200000 5.2350000}%
\special{ar 1000 1000 800 800  5.2800000 5.2950000}%
\special{ar 1000 1000 800 800  5.3400000 5.3550000}%
\special{ar 1000 1000 800 800  5.4000000 5.4150000}%
\special{ar 1000 1000 800 800  5.4600000 5.4750000}%
\special{ar 1000 1000 800 800  5.5200000 5.5350000}%
\special{ar 1000 1000 800 800  5.5800000 5.5950000}%
\special{ar 1000 1000 800 800  5.6400000 5.6550000}%
\special{ar 1000 1000 800 800  5.7000000 5.7150000}%
\special{ar 1000 1000 800 800  5.7600000 5.7750000}%
\special{ar 1000 1000 800 800  5.8200000 5.8350000}%
\special{ar 1000 1000 800 800  5.8800000 5.8950000}%
\special{ar 1000 1000 800 800  5.9400000 5.9550000}%
\special{ar 1000 1000 800 800  6.0000000 6.0150000}%
\special{ar 1000 1000 800 800  6.0600000 6.0750000}%
\special{ar 1000 1000 800 800  6.1200000 6.1350000}%
\special{ar 1000 1000 800 800  6.1800000 6.1950000}%
\special{ar 1000 1000 800 800  6.2400000 6.2550000}%
}}%
\put(12.8500,-13.1000){\makebox(0,0)[rb]{$\lambda_N$}}%
%
{\color[named]{Black}{%
\special{pn 8}%
\special{ar 1096 1006 106 66  3.1415927 6.2831853}%
}}%
%
{\color[named]{Black}{%
\special{pn 8}%
\special{ar 1196 1000 206 126  3.1415927 6.2831853}%
}}%
%
{\color[named]{Black}{%
\special{pn 8}%
\special{ar 1400 1000 404 236  3.1415927 6.2831853}%
}}%
\put(16.2500,-7.3000){\makebox(0,0)[rb]{$2r_{\lambda}$}}%
\put(8.8000,-7.5000){\makebox(0,0)[lb]{$r_{\lambda}$}}%
\put(6.8500,-7.5000){\makebox(0,0)[rb]{$r_{\lambda}/2$}}%
%
{\color[named]{Black}{%
\special{pn 8}%
\special{pa 1096 946}%
\special{pa 630 736}%
\special{fp}%
}}%
\put(19.0000,-11.8000){\makebox(0,0)[lb]{$\underbrace{\qquad \qquad \qquad \quad}$}}%
\put(12.4500,-4.7000){\makebox(0,0)[rb]{$\varGamma_{\lambda}(S_{p}^{\mu})$}}%
%
{\color[named]{Black}{%
\special{pn 8}%
\special{pa 1000 600}%
\special{pa 1000 500}%
\special{fp}%
}}%
\put(20.0000,-13.4500){\makebox(0,0)[lb]{Elements of }}%
\put(20.0000,-15.4500){\makebox(0,0)[lb]{$\sigma(S_{p}^{\mu})\setminus \{\lambda \}$\  or }}%
\put(20.0000,-17.5500){\makebox(0,0)[lb]{$\sigma(S_{p,N}^{\mu})\setminus \{\lambda_N \}$}}%
%
{\color[named]{Black}{%
\special{pn 20}%
\special{pa 1030 606}%
\special{pa 1006 600}%
\special{fp}%
\special{sh 1}%
\special{pa 1006 600}%
\special{pa 1066 634}%
\special{pa 1058 610}%
\special{pa 1074 594}%
\special{pa 1006 600}%
\special{fp}%
\special{pa 1006 600}%
\special{pa 1006 600}%
\special{fp}%
}}%
%
{\color[named]{Black}{%
\special{pn 4}%
\special{sh 1}%
\special{ar 2100 900 14 14 0  6.28318530717959E+0000}%
\special{sh 1}%
\special{ar 2300 800 14 14 0  6.28318530717959E+0000}%
\special{sh 1}%
\special{ar 2500 900 14 14 0  6.28318530717959E+0000}%
\special{sh 1}%
\special{ar 2500 900 14 14 0  6.28318530717959E+0000}%
}}%
%
{\color[named]{Black}{%
\special{pn 8}%
\special{pa 1000 1000}%
\special{pa 860 1190}%
\special{fp}%
}}%
%
{\color[named]{Black}{%
\special{pn 8}%
\special{pa 1110 1000}%
\special{pa 1166 1180}%
\special{fp}%
\special{pa 1176 1200}%
\special{pa 1176 1200}%
\special{fp}%
}}%
%
{\color[named]{Black}{%
\special{pn 13}%
\special{pa 1000 1000}%
\special{pa 1800 1000}%
\special{dt 0.045}%
}}%
%
{\color[named]{Black}{%
\special{pn 8}%
\special{pa 1206 880}%
\special{pa 1050 726}%
\special{fp}%
}}%
\end{picture}%
\end{center}
\caption{
Illustration of the conditions \eqref{eq:3a_in} and \eqref{eq:3a_out} in
Assumption~\ref{assump:approx_cond_single}. 
This figure depicts locations of an exact eigenvalue 
$\lambda \in \sigma(S_{p}^{\mu})$, the corresponding approximate eigenvalue
$\lambda_{N} \in \sigma(S_{p,N}^{\mu})$ and some other eigenvalues on the complex plane $\mathbf{C}$.  
The eigenvalue $\lambda_{N}$ is located inside the circle with center $\lambda$ and radius $r_{\lambda}/2$ 
(shown by the smaller dotted circle), 
and the other eigenvalues than $\lambda$ and $\lambda_{N}$ 
are outside the circle with center $\lambda$ and radius $2r_{\lambda}$
(shown by the larger dotted circle).
The directed contour $\varGamma_{\lambda}(S_{p}^{\mu})$
is used to prove the theorems in Sections~\ref{subsec:HillConvRate} and~\ref{subsec:localization}.
}
\label{fig:eigenvalues01}
\end{figure}

When Assumption~\ref{assump:approx_cond_single} or~\ref{assump:approx_cond_plural} 
are satisfied, there exist some convergent subsequences in $\bigcup_{N} \sigma(S_{p, N}^{\mu})$. 
Thus these assumptions are closely related to the convergence property of the eigenvalues of $S_{p, N}^{\mu}$. 
On this convergence property, the following two theorems are known \cite{bib:CurtisDeconinck2010}\cite{bib:ReedSimonMMMP-I}.
Theorem~\ref{thm:no-spurious-mode} gives so-called ``no-spurious-mode condition''.

\begin{thm}[{\cite[Theorem 9]{bib:CurtisDeconinck2010}}]
\label{thm:no-spurious-mode}
Let $\mathfrak{D} \subset \mathbf{C}$ be a compact set and 
$\{ \lambda_N \}$ be a sequence in $\mathfrak{D} $ with $\lambda_{N} \in \sigma(S_{p, N}^{\mu})$. 
Then for any $\varepsilon > 0$ there exists a positive integer $M$ such that the following holds true: 
for any $N$ with $N \geq M$, 
$\lambda_N$ is contained in the $\varepsilon$-neighborhood of some $\lambda \in \mathfrak{D}  \cap \sigma(S_{p}^{\mu})$.
\end{thm}

\begin{thm}[{\cite[Corollary of Lemma 10 and Theorem 16]{bib:CurtisDeconinck2010}}]
\label{thm:all_lambda_produced}
Assume that $S_{p}^{\mu}$ is self-adjoint. 
Then for any $\lambda \in \sigma(S_{p}^{\mu})$ 
there exists some subsequence $\{ \lambda_{N_{j}} \}$ of $\{ \lambda_{N} \}$ with $\lambda_{N} \in \sigma(S_{p,N}^{\mu})$ 
such that $\lambda_{N_{j}} \to \lambda$. 
\end{thm}

In Section~\ref{subsec:localization}, 
a sufficient condition for~\eqref{eq:3a_in} and~\eqref{eq:3a_out} in 
Assumption~\ref{assump:approx_cond_single} is given by 
Theorem~\ref{thm:local_2b_4a}.
We can verify this sufficient condition using 
Theorems~\ref{thm:no-spurious-mode}, \ref{thm:all_lambda_produced} and 
the following theorem for a concrete problem 
as shown in Section~\ref{sec:NumComp}.

\begin{thm}[Gershgorin's Theorem \cite{bib:VargaGersh}]
\label{thm:Gershgorin}
For any $n\times n$ matrix $V = (v_{ij})$ with $v_{ij} \in \mathbf{C}$, 
all eigenvalues of $V$ are contained in $\bigcup_{i=1}^{n} C_{i}$, 
where $C_1, \ldots, C_n$ are disks defined as
\begin{align}
C_{i} = \left\{ z \in \mathbf{C}\, \left| \, | z - v_{ii} | \leq \sum_{j \neq i} |v_{ij}| \right. \right\}.
\label{eq:GershDisc}
\end{align}
Moreover, 
each connected component of $\bigcup_{i=1}^{n} C_{i}$
contains as many eigenvalues of $V$ as 
the disks composing it.
\end{thm}
\noindent
This theorem is used to estimate the radius $r_{\lambda}$ 
in the conditions~\eqref{eq:3a_in} and~\eqref{eq:3a_out}, 
whereas Theorems~\ref{thm:no-spurious-mode} and~\ref{thm:all_lambda_produced} 
guarantee that the sequence of the approximate eigenvalues converges to an exact eigenvalue.

\subsection{A Priori Estimates of Convergence Rates}
\label{subsec:HillConvRate}

When Assumptions~\ref{assump:self-adjoint}, 2 and 3 
are satisfied, we can obtain a priori estimates of the convergence rates as follows. 

\begin{thm}\label{thm:main01}
Let Assumptions~\ref{assump:self-adjoint} 
and~\ref{assump:approx_cond_single} be satisfied. Then, 
when Assumption~\ref{assump:C_infty} is satisfied, namely $f_{j}$ being $C^{\infty}$, 
for a sufficiently large $N$ and
any positive integer $q$, there exists a positive constant $C_{1}$ 
such that
\begin{align}
 | \lambda - \lambda_N | \leq C_{1}\, N^{-q}, 
\label{eq:main_01}
\end{align}
where $C_{1}$ depends only on 
$p$, $f_{0},\ldots, f_{p-1}$, $\lambda$, $r_{\lambda}$, $q$
and an eigenvector $\phi$ of $S_{p}^{\mu}$ corresponding to $\lambda$. 
In addition, 
when Assumption~\ref{assump:analytic} is satisfied, namely $f_{j}$ being analytic, 
for a sufficiently large $N$ and 
any $\varepsilon$ with $0< \varepsilon < d$, there exists a positive constant $C_{2}$ 
such that
\begin{align}
 | \lambda - \lambda_N | \leq C_{2}\, N^{p+1/2} \exp \left( -\frac{2 \pi (d - \varepsilon)}{L} N \right), 
\label{eq:main_02}
\end{align}
where $C_{2}$ depends only on 
$p$, $f_{0},\ldots, f_{p-1}$, $\lambda$, $r_{\lambda}$, $d$, $\varepsilon$
and $\phi$. 
\end{thm}

\begin{thm}\label{thm:main03}
Let Assumptions~\ref{assump:self-adjoint} 
and~\ref{assump:approx_cond_plural} be satisfied. Then, 
when Assumption~\ref{assump:C_infty} is satisfied, namely $f_{j}$ being $C^{\infty}$, 
for a sufficiently large $N$, 
some choice of an index $i_{N} \in \{1,\ldots, k \}$ and 
any positive integer $q$,
there exists a positive constant $C_{1}$ 
such that
\begin{align}
 | \lambda - \lambda_{N, i_{N}} | \leq C_{1}\, N^{-q}, 
\label{eq:main_03}
\end{align}
where $C_{1}$ depends only on 
$p$, $f_{0},\ldots, f_{p-1}$, $\lambda$, $r_{\lambda}$, $q$ 
and an eigenvector $\phi$ of $S_{p}^{\mu}$ corresponding to $\lambda$. 
In addition, 
when Assumption~\ref{assump:analytic} is satisfied, namely $f_{j}$ being analytic, 
for a sufficiently large $N$, 
some choice of an index $i_{N} \in \{1,\ldots, k \}$ and 
any $\varepsilon$ with $0< \varepsilon < d$,
there exists a positive constant $C_{2}$
such that
\begin{align}
 | \lambda - \lambda_{N, i_{N}} | \leq C_{2}\, N^{2p+1} \exp \left( -\frac{2 \pi (d - \varepsilon)}{k L} N \right), 
\label{eq:main_04}
\end{align}
where $C_{2}$ depends only on 
$p$, $f_{0},\ldots, f_{p-1}$, $\lambda$, $r_{\lambda}$, $d$, $\varepsilon$ and $\phi$. 
\end{thm}

\subsection{A Posteriori Error Bounds}
\label{subsec:localization}

When Assumptions~\ref{assump:self-adjoint} and~\ref{assump:C_infty} 
are satisfied, we can explicitly obtain a posteriori estimates of the error bounds as follows. 
Here, for simplicity, only the case of Assumption~\ref{assump:approx_cond_single} is considered. 

\begin{thm}
\label{thm:local_2b_4a}
Let Assumptions~\ref{assump:self-adjoint} and~\ref{assump:C_infty} be satisfied. 
Assume that there exists $\zeta \in \mathbf{C}$ and $r>0$ such that, 
for any $N$, 
\begin{align}
\lambda_{N} \in B_{\zeta}(r)
\quad \text{and}\quad  
B_{\zeta}(9r) \cap \left( \sigma(S_{p, N}^{\mu}) \setminus \{ \lambda_{N} \} \right) = \emptyset,
\label{eq:local_2b}
\end{align}
where $\lambda_{N} \in \sigma(S_{p, N}^{\mu})$ and $B_{\zeta}(r)$ is defined by~\eqref{eq:DefDisk}. 
Then there uniquely exists $\lambda \in \sigma(S_{p}^{\mu}) \cap B_{\zeta}(r)$ satisfying 
\begin{align}
| \lambda - \lambda_{N} | 
\leq &
\left(
5 + \frac{3 |\zeta|}{r}
\right)
\frac{(2\pi N)^{p}}{L^{p+1/2}} \notag \\
& \cdot \sum_{j=0}^{p} 
\left(
\sum_{N < |l| < 2N} 
\sum_{m = l - N}^{l + N} 
\left| (\hat{f}_{j})_{m} (\hat{\phi}_{N})_{l-m} \right|
+
(2N+1)
\sum_{|m| \geq N} 
\left| (\hat{f}_{j})_{m} \right|
\right), 
\label{eq:bound}
\end{align}
where $f_{p} \equiv 1$, 
$\phi_{N}$ is an eigenvector corresponding to $\lambda_{N}$, 
and 
$(\hat{f}_{j})_{m}\ (m=0, \pm 1, \pm 2, \ldots)$ and 
$(\hat{\phi}_{N})_{n}\ (n=0, \pm 1, \pm 2, \ldots)$ 
are the Fourier coefficients of 
$f_{j}\ (j=0,\ldots, p)$ and 
$\phi_{N}$, respectively. 
\end{thm}

\section{Numerical Experiments}
\label{sec:NumComp}

In this section, we apply Hill's method described in Section~\ref{sec:HillsMethod} to Hill's operator \eqref{eq:HillEqSn}
of which the spectrum is exactly known,
and observe that computed results are consistent with Theorems~\ref{thm:main01},~\ref{thm:main03} 
and~\ref{thm:local_2b_4a} in Section~\ref{sec:DiffOpHill}. 


\subsection{Hill's Operator}

As an example for numerical experiments, we consider Hill's operator  \cite{bib:DeconinckKutz2006} defined by
\begin{align}
S_{2} = -\frac{\mathrm{d}^2}{\mathrm{d}x^2} + \{ 6 \ell^2\, \mathrm{sn}^2(x, \ell) -4 - \ell^2 \},
\label{eq:HillEqSn}
\end{align}
where $\mathrm{sn}(\cdot, \ell)$ is the Jacobian elliptic function with modulus $\ell\ (0\leq \ell < 1)$. 
It is known \cite{bib:DeconinckKutz2006} that the spectrum $\sigma(S_{2})$ is exactly given by 
\begin{align}
\sigma(S_{2}) = [\sigma_{\mathrm{a}}(\ell),\, -3] \cup [\sigma_{\mathrm{b}}(\ell),\, 0] \cup [\sigma_{\mathrm{c}}(\ell),\, +\infty),
\label{eq:exact_spec_S_2}
\end{align}
where
\begin{align}
& \text{(a)}\quad \sigma_{\mathrm{a}}(\ell) = \ell^2 - 2 - 2 \sqrt{1 - \ell^2 + \ell^4}, \notag \\
& \text{(b)}\quad \sigma_{\mathrm{b}}(\ell) = -3(1-\ell^2), \label{eq:eigen_abc}\\
& \text{(c)}\quad \sigma_{\mathrm{c}}(\ell) = \ell^2 - 2 + 2 \sqrt{1 - \ell^2 + \ell^4}. \notag
\end{align}
Figure~\ref{fig:HillEqSnSpectra} illustrates $\sigma(S_{2})$ for all $\ell$ with $0\leq \ell < 1$.

By definition~\eqref{eq:Expr_S_p_mu} of $S_p^{\mu}$, $p=2$ and $S_{2}^{\mu}$ is expressed as
\begin{align}
S_{2}^{\mu} 
& = - \frac{\mathrm{d^2}}{\mathrm{d}x^2} 
+ f_{1}(x) \frac{\mathrm{d}}{\mathrm{d}x} 
+ f_{0}(x)
\label{eq:S2mu}
\end{align}
with 
\begin{align}
f_{1}(x) = - 2 \mathrm{i} \mu \quad \text{and}\quad  
f_{0}(x) = \{ 6 \ell^2\, \mathrm{sn}^2(x, \ell) -4 - \ell^2 \} + \mu^2.
\end{align}
Although the sign of the highest order derivative 
$\mathrm{d}^2 / \mathrm{d}x^2$ in~\eqref{eq:S2mu}
is different from that in~\eqref{eq:Def_S_p_mu}, 
we can directly apply the results in Section~\ref{sec:DiffOpHill}
to $S_{2}^{\mu}$.
Note that $f_{0}$ is periodic with period $2K(\ell)$, 
where $K(\ell)$ is the complete elliptic integral of the first kind with modulus $\ell$.
Since $S_{2}^{\mu}$ is self-adjoint and $f_{0}$ is analytic with 
$d = K\left(\sqrt{1-\ell^2}\right) $ in \eqref{eq:DefD_d}, 
Assumptions 1 and 2b are satisfied. 

\begin{figure}[t]
\begin{center}
\begin{minipage}{0.7\linewidth}
\begin{center}
\includegraphics[width=0.7\linewidth]{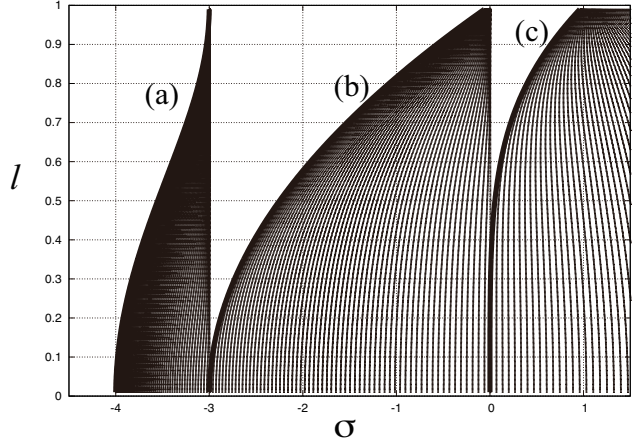}
\caption{Computed results of $( \sigma(S_{2}), \ell)$ for $S_{2}$ in \eqref{eq:HillEqSn} 
(a) $\sigma_{\mathrm{a}}(\ell) = \ell^2 - 2 - 2 \sqrt{1 - \ell^2 + \ell^4}$, 
(b) $\sigma_{\mathrm{b}}(\ell) = -3(1-\ell^2)$ and 
(c) $\sigma_{\mathrm{c}}(\ell) = \ell^2 - 2 + 2 \sqrt{1 - \ell^2 + \ell^4}$. }
\label{fig:HillEqSnSpectra}
\end{center}
\end{minipage}
\end{center}
\end{figure}

For later convenience, expand the periodic term $f_{0}(x) - \mu^{2}$ in the Fourier series form
\begin{align*}
f_{0}(x) - \mu^{2}
= b_{0} + \sum_{j = 1}^{\infty} b_{j} 
\left\{ \exp \left( \mathrm{i} \frac{\pi j x}{K(\ell)} \right) + \exp \left( - \mathrm{i} \frac{\pi j x}{K(\ell)} \right) \right\}, 
\end{align*}
where 
\begin{align}
b_{0}(\ell) = 6\left( 1 - \frac{E(\ell)}{K(\ell)} \right) -4 - \ell^2, \quad 
b_{j}(\ell) = - \frac{6\pi^2}{K(\ell)^2} \frac{j\, q(\ell)^j}{1-q(\ell)^{2j}}\quad (j = 1,2,\ldots), \label{eq:AN_02}
\end{align}
$E(\ell)$ is the complete elliptic integral of the second kind with modulus $\ell$, and 
\begin{align}
q(\ell) = \exp\left( -\pi \frac{K\left(\sqrt{1 - \ell^2}\right)}{K(\ell)} \right).\label{eq:AN_03}
\end{align}
Then, for $L = 2MK(\ell)$ with positive integer $M$, 
the eigenvalue problem $S_{2, N}^{\mu} \phi = \lambda \phi$ for $\phi \in L_{2}([0, L])_{\mathrm{per}}$
can be written in the form
\begin{align}
\hat{D}_{n}^{\mu}(\ell)\ (\hat{\phi}_{N})_{n} + 
\sum_{j = 1}^{N} b_{j}(\ell) \left( (\hat{\phi}_{N})_{n - M j} + (\hat{\phi}_{N})_{n + M j} \right) 
= \lambda (\hat{\phi}_{N})_{n} \quad
(-N \leq n \leq N)
\label{eq:S2mu_matrix}
\end{align}
with
\begin{align}
\hat{D}_{n}^{\mu}(\ell) = \left(\mu + \frac{\pi n}{MK(\ell)} \right)^2 + b_{0}(\ell)
\quad (n = -N, \ldots, N).
\label{eq:AN_01} 
\end{align}
Equation~\eqref{eq:S2mu_matrix} is a matrix eigenvalue problem of $2N+1$ dimension. 

For given $M$ and $N$, we can compute the eigenvalues of $S_{2, N}^{\mu}$ using Hill's method. 
Figure~\ref{fig:Error_a} shows the absolute errors of the computed eigenvalues 
corresponding to (a) in \eqref{eq:eigen_abc} for $M=2$ and $N=5, 10, \ldots, 40$.  
Here it should be noted that the boundaries (a), (b) and (c) in~\eqref{eq:eigen_abc} 
are obtained as the eigenvalues of $S_{p}^{\mu}$ in~\eqref{eq:S2mu} for $\mu = 0$. 
All programs used in this paper 
are written in C with quadruple-precision floating-point arithmetic 
and executed on a computer with the Sparc processor and the Fujitsu compiler C99.

\subsection{The Convergence Rate}
\label{subsec:NumConvRate}


The computed eigenvalues of $S_{2, N}^{0}$
approach to the exact ones corresponding to them 
with increase of $N$ as shown in Figure~\ref{fig:Error_a}.
Thus the conditions~\eqref{eq:3a_in} and~\eqref{eq:3a_out} in 
Assumption~\ref{assump:approx_cond_single} are satisfied 
for sufficiently large $N$, 
and the convergence rate~\eqref{eq:main_02}
in Theorem~\ref{thm:main01}
can be applied to the computed eigenvalues $\lambda_{N}$'s. 
In fact, the exponential decay of the error with $N$ in \eqref{eq:main_02}
can be found in Figure~\ref{fig:Error_a}.
Note that 
similar results to Figure~\ref{fig:Error_a} are obtained for 
the eigenvalues (b) and (c). 
From these, 
we can say that the computed results are consistent with 
\eqref{eq:main_02} in Theorem~\ref{thm:main01}.

For small values of $\ell$, 
we can show that the convergence rate~\eqref{eq:main_02} holds, 
even if the exact eigenvalues of $S_{2}^{0}$ are unknown, 
as follows. 
For that, 
first, consider~\eqref{eq:local_2b} which is the sufficient condition for~\eqref{eq:3a_in} and~\eqref{eq:3a_out} 
in Assumption~\ref{assump:approx_cond_single}. 
We can check~\eqref{eq:local_2b} by applying
Theorem~\ref{thm:Gershgorin} (Gershgorin's theorem) to the coefficient matrix 
in the left hand side of \eqref{eq:S2mu_matrix}, 
of which the diagonal elements are given by
$\hat{D}_{n}^{\mu}(\ell)$ in~\eqref{eq:AN_01} and 
the eigenvalues are real. 
For this matrix, 
Theorem~\ref{thm:Gershgorin} produces
\begin{align}
\sigma(S_{2,N}^{\mu}) 
\subset 
\bigcup_{n=-N}^{N}
I_n^{\mu}(\ell)
\label{eq:inclusion}
\end{align}
where
\begin{align}
& I_n^{\mu}(\ell)
= \left\{ 
z \in \mathbf{R}\, \left| \, | z - \hat{D}_n^{\mu}(\ell) | \leq r(\ell) \right.
\right\} \label{eq:inclusion_interval} \\
\intertext{with}
& r(\ell) = \frac{12 \pi^2}{K(\ell)^2} \frac{q(\ell)}{(1-q(\ell)^2)(1-q(\ell))^2}.
\label{eq:inclusion_interval_rad}
\end{align}
We can compute the center $\hat{D}_n^{\mu}(\ell)$ in~\eqref{eq:AN_01} 
and the radius $r(\ell)$ in \eqref{eq:inclusion_interval_rad} of 
the interval $I_n^{\mu}(\ell)$ which includes some eigenvalues in $\sigma(S_{2,N}^{\mu}) $. 
For example, when $\ell = 0.1$ and $\mu = 0$, we have 
\begin{align*}
r(0.1) = 0.030\cdots,
\end{align*}
and 
\begin{align*}
& \hat{D}_{0}^{0}(0.1) = -3.9799\cdots,\quad 
\hat{D}_{\pm 1}^{0}(0.1) = -2.9849\cdots, \quad \\
& \hat{D}_{\pm 2}^{0}(0.1) \doteqdot 0, \quad 
\hat{D}_{\pm 3}^{0}(0.1) = 4.9749\cdots.
\end{align*}
Then 
the intervals $I_{n}^{0}(0.1)\ (n=0,1,2,3)$ are mutually disjoint and 
the distances between their centers are larger than $10\, r(0.1)$. 
Moreover, from Theorem~\ref{thm:Gershgorin}, we can say that
$I_{0}^{0}(0.1)$ contains just one eigenvalue of $S_{2, N}^{\mu}$. 
Note that $\zeta$ and $r$ in~\eqref{eq:local_2b} correspond to $\hat{D}_{0}^{0}(0.1)$ and $r(0.1)$, 
respectively. 
Therefore the condition~\eqref{eq:local_2b} is satisfied for the interval $I_{0}^{0}(0.1)$. 
From these, 
the conditions~\eqref{eq:3a_in} and~\eqref{eq:3a_out}  
in Assumption~\ref{assump:approx_cond_single} are satisfied, and the convergence rate~\eqref{eq:main_02} holds for $I_{0}^{0}(0.1)$.

Also from Theorem~\ref{thm:Gershgorin}, we can show that 
$I_{1}^{0}(0.1)$ and $I_{2}^{0}(0.1)$ contain just two eigenvalues of $S_{2, N}^{\mu}$, 
respectively. 
Then the conditions~\eqref{eq:3b_in} and~\eqref{eq:3b_out} in Assumption~\ref{assump:approx_cond_plural} may be satisfied, and 
it was found that the numerical results are consistent with the convergence rate \eqref{eq:main_04}.
For further study on this case of Assumption~\ref{assump:approx_cond_plural}, 
we have to develop the verification method for the conditions~\eqref{eq:3b_in} and~\eqref{eq:3b_out}.

\subsection{The Error Bound}
\label{subsec:NumErrorBound}

Since the computed eigenvalues are accurate enough 
for sufficiently large $N$, as shown in Figure~\ref{fig:Error_a}, 
the condition~\eqref{eq:local_2b} in Theorem~\ref{thm:local_2b_4a}
is satisfied for sufficiently large $N$.
Then we can estimate the error bound~\eqref{eq:bound} in Theorem~\ref{thm:local_2b_4a} as follows. 
First, $\zeta$, $r$, $L$ and $(\hat{f}_{0})_{m}$ in~\eqref{eq:bound} are given by, respectively, 
\begin{align}
& \zeta = \hat{D}_{0}^{0}(\ell), \ 
r = r(\ell),\ 
L = 4K(\ell),\ \text{and}\ 
(\hat{f}_{0})_{m} = 
\begin{cases}
0 & (m \text{ is odd}), \\
\sqrt{L}\, b_{m/2} & (m \text{ is even}),
\end{cases}
\label{eq:boundparam}
\end{align}
for $|m| \geq 1$. Next, the infinite sum in~\eqref{eq:bound} can be bounded as 
\begin{align}
\sum_{|m| \geq N} \left| (\hat{f}_{0})_{m} \right|
& \leq 
2 \sqrt{L} \sum_{j = \lceil N/2 \rceil}^{\infty} | b_{j} |
\leq 
\frac{24\pi^{2}}{K(\ell)^{3/2}} \sum_{j = \lceil N/2 \rceil}^{\infty} \frac{j\, q(\ell)^{j}}{1- q(\ell)^2} \notag \\
& = 
\frac{24\pi^{2}}{K(\ell)^{3/2}} \frac{(\lceil N/2 \rceil (1-q(\ell)) + q(\ell))\, q(\ell)^{\lceil N/2 \rceil}}{(1-q(\ell)^{2}) (1-q(\ell))^{2}}. 
\label{eq:boundbound}
\end{align}
Then it follows from the inequality~\eqref{eq:bound} that 
\begin{align}
| \lambda - \lambda_{N} | 
\leq &
\left(
5 + \frac{3 |\zeta|}{r}
\right)
\frac{(2\pi N)^{p}}{L^{p+1/2}}
\left(
\sum_{N < |l| < 2N} 
\sum_{m = l - N}^{l + N} 
\left| (\hat{f}_{0})_{m} (\hat{\phi}_{N})_{l-m} \right|
\right. \notag \\
& +
\left.
(2N+1)
\frac{24\pi^{2}}{K(\ell)^{3/2}} \frac{(\lceil N/2 \rceil (1-q(\ell)) + q(\ell))\, q(\ell)^{\lceil N/2 \rceil}}{(1-q(\ell)^{2}) (1-q(\ell))^{2}} 
\right), 
\label{eq:bound_mdfd}
\end{align}
where $\zeta$, $r$, $L$ and $(\hat{f}_{0})_{m}$ are given by \eqref{eq:boundparam}. 
The computed results of the right hand side of \eqref{eq:bound_mdfd}
are shown in Figure~\ref{fig:Bound_a}. 
The absolute errors in Figure~\ref{fig:Error_a} are bounded by 
the corresponding error bounds in Figure~\ref{fig:Bound_a}. 
Thus the numerical results are consistent with
Theorem~\ref{thm:local_2b_4a}.

For small values of $\ell$, 
we can obtain the explicit error bound~\eqref{eq:bound_mdfd}, 
even if the exact eigenvalues of $S_{2}^{0}$ are unknown.
For example, when $\ell =0.1$ and $\mu = 0$, 
the interval $I_{0}^{0}(0.1)$ in \eqref{eq:inclusion_interval}
satisfies the condition~\eqref{eq:local_2b} as shown in Section \ref{subsec:NumConvRate}. 
Then we can conclude that $I_{0}^{0}(0.1)$ contains one exact eigenvalue of $S_{2}^{0}$, 
and the corresponding error is bounded by~\eqref{eq:bound_mdfd}. 
Figure~\ref{fig:Bound_a} shows that, when $\ell =0.1$, the error for $N  \geq 10$ is less than $2.37\times 10^{-8}$.

\begin{figure}
\begin{center}
\begin{minipage}[t]{0.45\linewidth}
\begin{center}
\includegraphics[width=\linewidth]{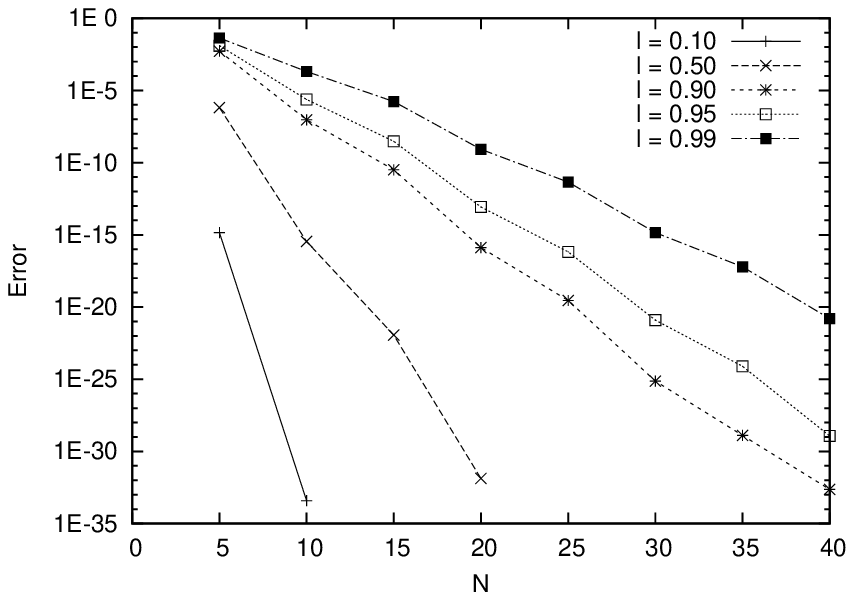}
\caption{The absolute errors of the approximate eigenvalues of $\sigma(S_{2,N}^{0})$ corresponding to the eigenvalues (a).}
\label{fig:Error_a}
\end{center}
\end{minipage}\quad 
\begin{minipage}[t]{0.45\linewidth}
\begin{center}
\includegraphics[width=\linewidth]{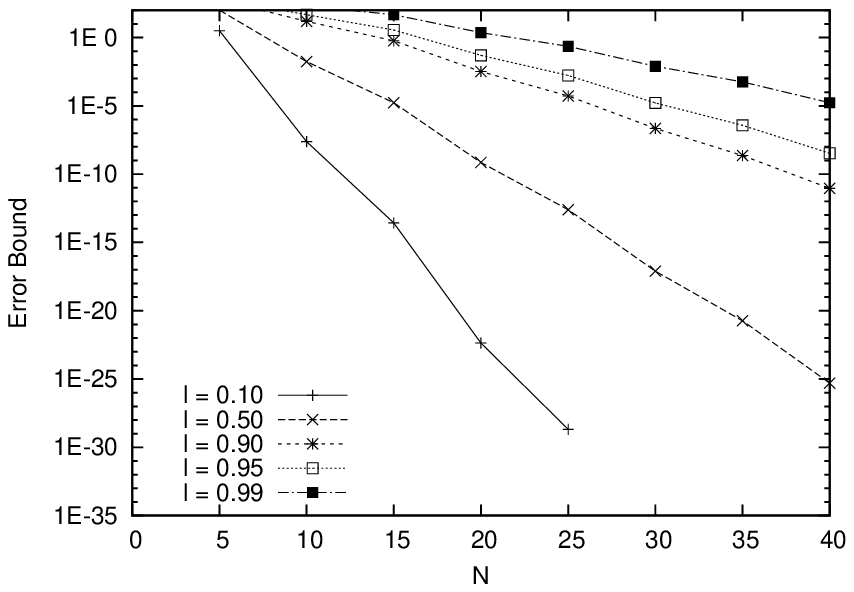}
\caption{The error bounds \eqref{eq:bound_mdfd} of the approximate eigenvalues of $\sigma(S_{2,N}^{0})$ corresponding to the eigenvalues (a).}
\label{fig:Bound_a}
\end{center}
\end{minipage}
\end{center}
\end{figure}

\section{Proofs}
\label{sec:GenConv}

This section summarizes proofs of 
the theorems in Sections~\ref{subsec:HillConvRate} and~\ref{subsec:localization}. 
Let $(\cdot\, ,\, \cdot)_{2}$ and $\| \cdot \|_{2}$ be the inner product of 
$L_2([0,L])_{\mathrm{per}}$ and the norm induced by it, respectively.
We also use $\| \cdot \|_{2}$ as the operator norm of linear operators on 
$L_2([0,L])_{\mathrm{per}}$ with respect to the norm $\| \cdot \|_{2}$. 

\subsection{Proof of Theorem~\ref{thm:main01}}
\label{subsec:main01_02}

To prove the theorem, 
we use an appropriate eigenvector $\phi_{N}$ of the approximate operator $S_{p, N}^{\mu}$ 
generated by an eigenvector $\phi$ of $S_{p}^{\mu}$. 
On Assumption~\ref{assump:approx_cond_single}, 
we can employ the projection operator $P_{\lambda}(S_{p, N}^{\mu})$ 
to the eigenspace corresponding to $\lambda_{N} \in \sigma(S_{p, N}^{\mu})$
defined by 
\begin{align}
 P_{\lambda}(S_{p, N}^{\mu}) = 
\frac{1}{2\pi \mathrm{i}} \oint_{\varGamma_{\lambda}(S_{p}^{\mu})} R_{\nu}(S_{p, N}^{\mu})\, \mathrm{d} \nu, 
\label{eq:DunfordIntP}
\end{align}
where $R_{\nu}(S_{p, N}^{\mu}) = (\nu I - S_{p, N}^{\mu})^{-1}$ and
$\varGamma_{\lambda}(S_{p}^{\mu})$
is the directed circle with center $\lambda$ and radius $r_{\lambda}$ 
shown by Figure~\ref{fig:eigenvalues01}. 
The integral in \eqref{eq:DunfordIntP} is the Dunford integral. 
See e.g.~Theorem XII.5 in \cite{bib:ReedSimonMMMP-IV}. 
For later use, it should be noted that 
\begin{align}
l(\varGamma_{\lambda}(S_{p}^{\mu})) = 2 \pi r_{\lambda}, 
\label{eq:contour_estimate_1}
\end{align}
where $l(\varGamma_{\lambda}(S_{p}^{\mu}))$ is the length of $\varGamma_{\lambda}(S_{p}^{\mu})$, 
and
\begin{align}
&\qquad |\nu - \lambda| = r_{\lambda}, \label{eq:contour_estimate_2}\\ 
& r_{\lambda}/2 \leq |\nu - \lambda_N | \leq 3r_{\lambda}/2. \label{eq:contour_estimate_3}
\end{align}
for any $\nu \in \varGamma_{\lambda}(S_{p}^{\mu})$. 



Then we take an eigenvector $\phi$ with $\| \phi \|_{2} = 1$ corresponding to $\lambda$, 
and using $\phi$ and the projection $P_{\lambda_{N}}(S_{p, N}^{\mu})$ for $\lambda_{N} \in \sigma(S_{p, N}^{\mu})$, 
we can obtain an eigenvector $\phi_{N}$ with $\| \phi_{N} \|_{2} = 1$ corresponding to $\lambda_{N}$. 
The feasibility of such procedure is guaranteed by the following lemma, 
which is based on the idea in \cite{bib:AtkinsonEigen1967}. 

\begin{lem}
\label{lem:make_approx_eigenvec_2a}
Let $S_{p}^{\mu}$ satisfy 
Assumptions~\ref{assump:self-adjoint} and~\ref{assump:C_infty}, 
$\lambda \in \sigma(S_{p}^{\mu})$ and $\lambda_{N} \in \sigma(S_{p, N}^{\mu})$
satisfy Assumption~\ref{assump:approx_cond_single}. 
Then for an eigenvector $\phi$ 
corresponding to $\lambda$ we have
\begin{align}
\| \phi - P_{\lambda_N}(S_{p, N}^{\mu}) \phi \|_{2}  
\leq 
\frac{2}{r_{\lambda}} \| (S_{p}^{\mu} - S_{p, N}^{\mu}) \phi \|_{2}. 
\label{eq:EstimByResol_2}
\end{align}
Furthermore, if $\| \phi \|_{2} = 1$ and 
there exists a positive integer $N_{0}$ such that for any $N > N_{0}$ 
\begin{align}
\frac{2}{r_{\lambda}} \| (S_{p}^{\mu} - S_{p, N}^{\mu}) \phi \|_{2} < 1, 
\label{eq:less_than_one}
\end{align}
we have $P_{\lambda_{N}}(S_{p, N}^{\mu}) \phi \neq 0$ for $N > N_{0}$.
\end{lem}

\begin{proof}
Under Assumption 2a, 
we can use the contour $\varGamma_{\lambda}(S_{p}^{\mu})$ to define $P_{\lambda}(S_{p}^{\mu})$ and $P_{\lambda_{N}}(S_{p, N}^{\mu})$. 
By the residue theorem, we have $P_{\lambda}(S_{p}^{\mu}) \phi = \phi$. 
Then it follows from \eqref{eq:contour_estimate_1} that 
\begin{align}
 \| \phi - P_{\lambda_N}(S_{p, N}^{\mu}) \phi \|_{2} 
& =  \| P_{\lambda}(S_{p}^{\mu}) \phi - P_{\lambda_N}(S_{p, N}^{\mu}) \phi \|_{2} \notag \\
& =  \left \|  \frac{1}{2\pi \mathrm{i}} 
\oint_{\varGamma_{\lambda}(S_{p}^{\mu})} (R_{\nu}(S_{p}^{\mu})- R_{\nu}(S_{p, N}^{\mu}))\phi\,  \mathrm{d}\nu \right \|_{2} \notag \\
& \leq r_{\lambda} \max_{\nu \in \varGamma_{\lambda}(S_{p}^{\mu})} \| (R_{\nu}(S_{p}^{\mu})- R_{\nu}(S_{p, N}^{\mu})) \phi \|_{2}.
\label{eq:EstimByResol}
\end{align}
As for $\| (R_{\nu}(S_{p}^{\mu})- R_{\nu}(S_{p, N}^{\mu})) \phi \|_{2}$ in \eqref{eq:EstimByResol}, we have
\begin{align}
\| (R_{\nu}(S_{p}^{\mu})- R_{\nu}(S_{p, N}^{\mu})) \phi \|_{2}
& = \| (\nu I - S_{p, N}^{\mu})^{-1} (S_{p}^{\mu} - S_{p, N}^{\mu}) (\nu I - S_{p}^{\mu})^{-1} \phi \|_{2} \notag \\
& = \| (\nu I - S_{p, N}^{\mu})^{-1} (S_{p}^{\mu} - S_{p, N}^{\mu}) (\nu - \lambda)^{-1} \phi \|_{2} \notag \\
& \leq | \nu - \lambda |^{-1} \, \| (\nu I - S_{p, N}^{\mu})^{-1} \|_{2}\,  \| (S_{p}^{\mu} - S_{p, N}^{\mu}) \phi \|_{2} \notag \\
& =  | \nu - \lambda |^{-1} \, \mathrm{dist}(\nu, \sigma(S_{p, N}^{\mu}))^{-1}\, \| (S_{p}^{\mu} - S_{p, N}^{\mu}) \phi \|_{2}, 
\label{eq:ResolCalc}
\end{align}
where $\mathrm{dist}(\nu, \sigma(S_{p, N}^{\mu}))$ is 
the distance between $\nu \in \mathbf{C}$ and $\sigma (S_{p, N}^{\mu})$ defined as
\[
\mathrm{dist}(\nu, \sigma(S_{p, N}^{\mu})) = \displaystyle \inf_{\xi \in \sigma (S_{p, N}^{\mu})} | \nu - \xi |.
\]
The last equality in \eqref{eq:ResolCalc} follows from the self-adjointness of $S_{p, N}^{\mu}$.
Furthermore, by \eqref{eq:contour_estimate_2} and \eqref{eq:contour_estimate_3} we have
\begin{align}
| \nu - \lambda | = r_{\lambda}, \label{eq:mu_lambda}
\end{align}
for any $\nu \in \varGamma_{\lambda}(S_{p}^{\mu})$ and 
\begin{align}
\min_{\nu \in \varGamma_{\lambda}(S_{p}^{\mu})} \mathrm{dist}(\nu, \sigma(S_{p, N}^{\mu})) 
= \min_{\nu \in \varGamma_{\lambda}(S_{p}^{\mu})} |\nu - \lambda_N | 
\geq r_{\lambda}/2. \label{eq:mu_lambda_N}
\end{align} 
Combining \eqref{eq:ResolCalc}, \eqref{eq:mu_lambda} and \eqref{eq:mu_lambda_N} we have
\begin{align}
\max_{\nu \in \varGamma_{\lambda}(S_{p}^{\mu})} \| (R_{\nu}(S_{p}^{\mu})- R_{\nu}(S_{p, N}^{\mu})) \phi \|_{2} 
\leq \frac{2}{r_{\lambda}^{2}} \| (S_{p}^{\mu} - S_{p, N}^{\mu}) \phi \|_{2}. 
\label{eq:ResolMaxEstim}
\end{align}
Then we can deduce the conclusion \eqref{eq:EstimByResol_2} 
from \eqref{eq:EstimByResol} and \eqref{eq:ResolMaxEstim}.

Finally, if $\| \phi \|_{2} = 1$ and \eqref{eq:less_than_one} holds for $N > N_{0}$, 
it follows from \eqref{eq:EstimByResol_2} that $\| \phi - P_{\lambda_N}(S_{p, N}^{\mu}) \phi \|_{2}  < 1$.
Hence $P_{\lambda_N}(S_{p, N}^{\mu}) \phi \neq 0$ holds true.
\end{proof}

Here we present a key lemma for the estimate of the convergence rate. 
Part of this lemma is also based on the idea in \cite{bib:AtkinsonEigen1967}. 

\begin{lem}\label{eq:StrResolEstim}
Let $S_{p}^{\mu}$ satisfy 
Assumptions~\ref{assump:self-adjoint} and~\ref{assump:C_infty}, 
$\lambda \in \sigma(S_{p}^{\mu})$ and $\lambda_{N} \in \sigma(S_{p, N}^{\mu})$
satisfy Assumption~\ref{assump:approx_cond_single}. 
Furthermore, for an eigenvector $\phi$ corresponding to $\lambda$ 
with $\| \phi \|_{2} = 1$, assume that 
there exists a positive integer $N_{0}$ such that for any $N > N_{0}$ 
\begin{align}
\frac{2}{r_{\lambda}} \| (S_{p}^{\mu} - S_{p, N}^{\mu}) \phi \|_{2} < 1.
\label{eq:A-ANphi_less_than_1_2a}
\end{align}
Then for $N > N_{0}$ we have
\begin{align}
| \lambda - \lambda_N | 
\leq
\left(
3 + \frac{8 |\lambda|}{r_{\lambda}}
\right)
\| (S_{p}^{\mu} - S_{p, N}^{\mu}) \phi \|_{2}.
\label{eq:GeneralRate2a}
\end{align}
\end{lem}

\begin{proof}
By Lemma~\ref{lem:make_approx_eigenvec_2a}, for $N>N_{0}$ we can define 
\begin{align}
\phi_N = \frac{1}{\| P_{\lambda_N}(S_{p, N}^{\mu}) \phi \|_{2}} P_{\lambda_N}(S_{p, N}^{\mu}) \phi.
\label{eq:MakingEigenvec}
\end{align}
Then $\phi_{N}$ is an eigenvector corresponding to $\lambda_{N}$ with $\| \phi_{N} \|_{2} = 1$. 
Then noting $\lambda = (S_{p}^{\mu} \phi, \phi)_{2}$ and $\lambda_N = (S_{p, N}^{\mu} \phi_N, \phi_N)_{2}$ we have
\begin{align}
\lambda - \lambda_N
& =  (S_{p}^{\mu}\phi, \phi)_{2} - (S_{p, N}^{\mu} \phi_N,  \phi_N)_{2} \notag \\
& =  (S_{p}^{\mu}\phi, \phi - \phi_N)_{2}  +  ((S_{p}^{\mu} - S_{p, N}^{\mu}) \phi, \phi_N)_{2} + (\phi - \phi_N, S_{p, N}^{\mu} \phi_N)_{2},
\label{eq:LamdaEstim0}
\end{align}
where the last equality follows from the self-adjointness of $S_{p, N}^{\mu}$. 
Therefore we have
\begin{align}
|\lambda - \lambda_N|
& \leq (|\lambda| + |\lambda_{N}| )\, \| \phi - \phi_N \|_{2} + \| (S_{p}^{\mu} - S_{p, N}^{\mu}) \phi \|_{2}. 
\label{eq:LamdaEstim1}
\end{align}
Here $\| \phi - \phi_N \|_{2}$ is bounded by
\begin{align}
\|  \phi - \phi_N \|_{2} 
& \leq \| \phi - P_{\lambda_N}(S_{p, N}^{\mu})\phi  \|_{2} + \| P_{\lambda_N}(S_{p, N}^{\mu})\phi - \phi_N  \|_{2}. 
\label{eq:split_phi_phi_N}
\end{align}
Since $\| P_{\lambda_N}(S_{p, N}^{\mu})\phi - \phi_N  \|_{2} $ in \eqref{eq:split_phi_phi_N} is bounded as
\begin{align*}
\| P_{\lambda_N}(S_{p, N}^{\mu})\phi - \phi_N\|_{2} 
& = \left\| \left( 1 - \frac{1}{\| P_{\lambda_N}(S_{p, N}^{\mu}) \phi \|_{2}}\right) P_{\lambda_N}(S_{p, N}^{\mu}) \phi \right \|_{2} \\
& = \left | \, \| P_{\lambda_N}(S_{p, N}^{\mu}) \phi \|_{2}  - 1\, \right | \\
& \leq \| P_{\lambda_N}(S_{p, N}^{\mu}) \phi - \phi \|_{2},
\end{align*}
we have
\begin{align}
\|  \phi - \phi_N \|_{2} 
& \leq 2 \| \phi - P_{\lambda_N}(S_{p, N}^{\mu}) \phi \|_{2}.
\label{eq:phi_phi_N_1}
\end{align}
Applying the inequality \eqref{eq:EstimByResol_2} of 
Lemma~\ref{lem:make_approx_eigenvec_2a} to \eqref{eq:phi_phi_N_1}, 
we have
\begin{align}
\|  \phi - \phi_N \|_{2} 
& \leq
\frac{4}{r_{\lambda}} \| (S_{p}^{\mu} - S_{p, N}^{\mu}) \phi \|_{2}.
\label{eq:phi_phi_N_2}
\end{align}
Then combining \eqref{eq:LamdaEstim1} and \eqref{eq:phi_phi_N_2}, we have
\begin{align*}
| \lambda - \lambda_N | 
\leq
\left(
\frac{4}{r_{\lambda}} (|\lambda| + |\lambda_{N}|) + 1
\right)
\| (S_{p}^{\mu} - S_{p, N}^{\mu}) \phi \|_{2}.
\end{align*}
Finally, noting $|\lambda_{N}| \leq | \lambda - \lambda_{N} | + |\lambda| \leq r_{\lambda} / 2 + |\lambda|$, we have the conclusion.
\end{proof}


\begin{rem}
\label{rem:SelfAd}
Assumption~\ref{assump:self-adjoint}, i.e.~the self-adjointness of $S_{p}^{\mu}$ is used 
(i) to estimate the norm of the resolvent of the approximate operator as 
$\| (\nu I - S_{p, N}^{\mu})^{-1} \|_{2}\ = \mathrm{dist}(\nu, \sigma(S_{p, N}^{\mu}))^{-1}$
in \eqref{eq:ResolCalc}, 
(ii) to make an approximate eigenvector as \eqref{eq:MakingEigenvec} and 
(iii) to deduce the equality \eqref{eq:LamdaEstim0}. 
\end{rem}

What remains is to estimate $\| (S_{p}^{\mu} - S_{p,N}^{\mu}) \phi \|_{2}$ in the RHS of \eqref{eq:GeneralRate2a}.
Since 
\begin{align*}
\| (S_{p}^{\mu} - S_{p,N}^{\mu}) \phi \|_{2} 
& = \| (S_{p}^{\mu}- \hat{P}_N S_{p}^{\mu}) \phi +(\hat{P}_N S_{p}^{\mu} - \hat{P}_N S_{p}^{\mu} \hat{P}_N) \phi \|_{2} \notag \\
& \leq \| (I - \hat{P}_N) S_{p}^{\mu} \phi \|_{2}  + \| \hat{P}_N S_{p}^{\mu} (I - \hat{P}_N) \phi \|_{2} \notag \\
& \leq ( |\lambda|  + \| \hat{P}_N S_{p}^{\mu}\|_{2} ) \| (I - \hat{P}_N) \phi \|_{2}, 
\end{align*}
we consider estimation of $\| \hat{P}_N S_{p}^{\mu} \|_{2}$ and $\| (I - \hat{P}_N) \phi \|_{2}$. 
Their estimates are given by the following lemmas, 
whose proofs are shown in Appendix~\ref{sec:ProofOrders}. 
It suffices to consider the case of Assumption~\ref{assump:C_infty} for $\| \hat{P}_N S_{p}^{\mu} \|_{2}$, 
whereas the both cases of Assumptions~\ref{assump:C_infty} and~\ref{assump:analytic} 
need to be considered for $\| (I - \hat{P}_N) \phi \|_{2}$
which depends on the smoothness of $\phi$. 

\begin{lem}
\label{lem:norm_P_N_S}
Let $S_{p}^{\mu}$ satisfy Assumption~\ref{assump:C_infty}. 
Then there exits a positive real constant $C$ depending only on $p$ and $f_{0}, f_{1}, \ldots, f_{p-1}$ such that
\[
\| \hat{P}_N S_{p}^{\mu} \|_{2} \leq C\, N^{p + 1/2}. 
\]
\end{lem}

\begin{lem}
\label{lem:ProjEstim_C_infty}
Let $S_{p}^{\mu}$ satisfy Assumption~\ref{assump:C_infty} and 
$\phi \in H^{p}([0,L])_{\mathrm{per}}$ be an eigenvector of $S_{p}^{\mu}$ corresponding to 
$\lambda \in \sigma(S_{p}^{\mu} )$ with $\|\phi\|_{2}=1$. 
Then for any positive integer $q$
there exists a positive constant $C$ depending only on $\phi$ and $q$
such that 
\begin{align*}
\| (I - \hat{P}_N) \phi \|_{2} \leq C\, (1 + N)^{-q}.
\end{align*}
\end{lem}

\begin{lem}
\label{lem:ProjEstim_analytic}
Let $S_{p}^{\mu}$ satisfy Assumption~\ref{assump:analytic} for $d>0$ and 
$\phi \in H^{p}([0,L])_{\mathrm{per}}$ be an eigenvector of $S_{p}^{\mu}$ corresponding to 
$\lambda \in \sigma(S_{p}^{\mu} )$ with $\|\phi\|_{2}=1$. 
Then for any $\varepsilon $ with $0 < \varepsilon < d$ 
there exists a positive constant $C$ depending only on $\phi$, $d$ and $\varepsilon$
such that
\begin{align*}
\| (I - \hat{P}_N) \phi \|_{2} \leq C\, \exp\left( - \frac{2\pi\, (d - \varepsilon )}{L}\, N \right),
\end{align*}
where $L$ is the period of the coefficient functions of $S_{p}^{\mu}$. 
\end{lem}

Here we prove Theorem~\ref{thm:main01}. 

\renewcommand{\proofname}{Proof of Theorem~\ref{thm:main01}}
\begin{proof}
Since it follows from Lemmas~\ref{lem:norm_P_N_S} and~\ref{lem:ProjEstim_C_infty} 
that \\ $\| (S_{p}^{\mu} - S_{p, N}^{\mu})\phi \|_{2} \to 0$ as $N \to \infty$, 
the condition \eqref{eq:A-ANphi_less_than_1_2a} in Lemma~\ref{eq:StrResolEstim} is satisfied. 
Therefore we can deduce the conclusion~\eqref{eq:main_01}
from 
Lemmas~\ref{eq:StrResolEstim}, \ref{lem:norm_P_N_S} and~\ref{lem:ProjEstim_C_infty}, 
and~\eqref{eq:main_02}
from 
Lemmas~\ref{eq:StrResolEstim}, \ref{lem:norm_P_N_S} and~\ref{lem:ProjEstim_analytic}. 
\end{proof}
\renewcommand{\proofname}{Proof}

\subsection{Proof of Theorem~\ref{thm:main03}}
\label{subsec:main03_04}

We use almost the same methods as the ones of Lemmas~\ref{lem:make_approx_eigenvec_2a} and~\ref{eq:StrResolEstim}.
On Assumption~\ref{assump:approx_cond_plural},
we consider $P_{\varLambda_N}(S_{p, N}^{\mu}) \phi$
for an eigenvector $\phi$ corresponding to $\lambda \in \sigma(S_{p, N}^{\mu})$, 
where $\varLambda_N = \{ \lambda_{N,1}, \ldots, \lambda_{N,k} \}$, 
\begin{align}
P_{\varLambda_{N}}(S_{p, N}^{\mu}) = 
\frac{1}{2\pi \mathrm{i}} \oint_{\varGamma_{\Lambda}(S_{p}^{\mu})} R_{\mu}(S_{p, N}^{\mu})\, \mathrm{d} \mu, 
\label{eq:DunfordIntP_pl}
\end{align}
and $\varGamma_{\Lambda}(S_{p}^{\mu})$ is a directed circle with center $\lambda$ and radius $r_{\lambda}$, 
which contains $\varLambda_N$. 
Here it should be noted that 
\eqref{eq:contour_estimate_1}, \eqref{eq:contour_estimate_2} and~\eqref{eq:contour_estimate_3}
can be used for Assumption~\ref{assump:approx_cond_plural}
when $\lambda_{N}$ is replaced by $\lambda_{N,i}\ (i=1,2,\ldots,k)$. 
In this case $P_{\varLambda_N}(S_{p, N}^{\mu}) \phi$ 
is not an eigenvector corresponding to a single eigenvalue of $S_{p, N}^{\mu}$
but a linear combination of eigenvectors corresponding to $\lambda_{N, 1},\ldots , \lambda_{N, k}$. 
Then we can estimate the convergence rate of one of them 
after appropriate rearrangement of them explained later. 

\begin{lem}
\label{lem:make_approx_eigenvec_2b}
Let $S_{p}^{\mu}$ satisfy 
Assumptions~\ref{assump:self-adjoint} and~\ref{assump:C_infty},  
$\lambda \in \sigma(S_{p}^{\mu})$ and $\lambda_{N, 1},\ldots, \lambda_{N, k} \in \sigma(S_{p, N}^{\mu})$
satisfy Assumption~\ref{assump:approx_cond_plural}. 
Then for an eigenvector $\phi$ 
corresponding to $\lambda$ we have
\begin{align}
\| \phi - P_{\varLambda_N}(S_{p, N}^{\mu}) \phi \|_{2}  
\leq 
\frac{2}{r_{\lambda}} \| (S_{p}^{\mu} - S_{p, N}^{\mu}) \phi \|_{2},
\label{eq:EstimByResol_2_2b}
\end{align}
where $\varLambda_N = \{ \lambda_{N,1}, \ldots, \lambda_{N,k} \}$.
Furthermore, if $\| \phi \|_{2} = 1$ and 
there exists a positive integer $N_{0}$ such that for any $N > N_{0}$ 
\begin{align}
\frac{2}{r_{\lambda}} \| (S_{p}^{\mu} - S_{p, N}^{\mu}) \phi \|_{2} < 1, 
\label{eq:less_than_one_2b}
\end{align}
we have $P_{\varLambda_{N}}(S_{p, N}^{\mu}) \phi \neq 0$ for $N > N_{0}$.
\end{lem}

\begin{proof}
Noting the contour $\varGamma_{\lambda}(S_{p}^{\mu}) = \partial B_{\lambda}(r_{\lambda})$
surrounds $\lambda$ and $\lambda_{N,1}, \ldots, \lambda_{N,k}$ in the case of Assumption~\ref{assump:approx_cond_plural}, 
we can use $\varGamma_{\lambda}(S_{p}^{\mu})$ to define $P_{\lambda}(S_{p}^{\mu})$ and $P_{\varLambda_{N}}(S_{p, N}^{\mu})$. 
Then we can obtain the conclusions in a similar manner to Lemma~\ref{lem:make_approx_eigenvec_2a}.  
\end{proof}

\begin{lem}
\label{eq:StrResolEstim_02}
Let $S_{p}^{\mu}$ satisfy 
Assumptions~\ref{assump:self-adjoint} and~\ref{assump:C_infty},  
$\lambda \in \sigma(S_{p}^{\mu})$ and $\lambda_{N, 1},\ldots, \lambda_{N, k} \in \sigma(S_{p, N}^{\mu})$
satisfy Assumption~\ref{assump:approx_cond_plural}. 
Furthermore, for an eigenvector $\phi$ corresponding to $\lambda$ 
with $\| \phi \|_{2} = 1$, assume that 
there exists a positive integer $N_{0}$ such that for any $N > N_{0}$ 
\begin{align}
\frac{2}{r_{\lambda}} \| (S_{p}^{\mu} - S_{p, N}^{\mu}) \phi \|_{2} < 1.
\label{eq:A-ANphi_less_than_1_2b}
\end{align}
Then for $N > N_{0}$ 
there exist 
a choice of an index $i_{N} \in \{1,\ldots, k \}$ and 
an integer $k'_{N}$ with $2\leq k'_{N} \leq k$
such that 
\begin{align}
\left| \lambda - \lambda_{N, i_{N}} \right| 
\leq
\left\{
K_{\lambda, N}\, \| (S_{p}^{\mu} - S_{p, N}^{\mu}) \phi \|_{2}
\right\}^{1/k'_{N}}, 
\label{eq:GeneralRate2b}
\end{align}
where
\begin{align*}
K_{\lambda, N} 
& = 
(4 |\lambda| + r_{\lambda} ) 
\left[
k^2
\max \left\{ 1, \left(|\lambda| + \frac{r_{\lambda}}{2}\right)^{2k} \right\} 
\max \left\{1, \|S_{p, N}^{\mu}\|_{2}^{k} \right\} 
 \right. \\
& \left. 
\quad +\frac{2}{r_{\lambda}} 
\max \left\{ 1, \left( \| S_{p, N}^{\mu} \|_{2} + | \lambda | + \frac{r_{\lambda}}{2} \right)^{k} \right\}
\right]
+ \max \left\{1, \left( \frac{r_{\lambda}}{2} \right)^{k} \right\}.
\end{align*}
In particular, 
if there exists $N_{1} \geq N_{0}$ such that $K_{\lambda, N}\, \| (S_{p}^{\mu} - S_{p, N}^{\mu}) \phi \|_{2} < 1$ for $N > N_{1}$, 
the integer $k'_{N}$ in \eqref{eq:GeneralRate2b} can be replaced by $k$.
\end{lem}

\begin{proof}
If there exist $i, j \in \{1,\ldots, k \}$ with $i\neq j$ such that $\lambda_{N,i} = \lambda_{N,j}$, 
we reserve only one of them, and consequently we obtain the set of indices $\Omega_{N}$ such that
$\{ \lambda_{N,i} \mid i \in \Omega_{N}\}$ is identical to $\{ \lambda_{N, 1},\ldots , \lambda_{N, k}\}$ as a set 
and the elements of it are mutually distinct. 
Then we renumber the elements of $\{ \lambda_{N,i} \mid i \in \Omega_{N}\}$ and let it be denoted by 
$\{ \lambda_{N,1},\ldots, \lambda_{N, k'_{N}} \}$, where $k'_{N} = \#\Omega_{N}$. 
Without loss of generality we can assume that $2 \leq k'_{N}$. 

It follows from Lemma~\ref{lem:make_approx_eigenvec_2b} that
$P_{\varLambda_N}(S_{p, N}^{\mu}) \phi $ is a nonzero linear combination of eigenvectors $\phi_{N, 1},\ldots , \phi_{N, k'_{N}}$ 
corresponding to $\lambda_{N, 1},\ldots , \lambda_{N,  k'_{N}}$ for  $N>N_{0}$.
Then there exists a sequence of coefficients $a_{N, 1}, \ldots, a_{N, k'_{N}} \in \mathbf{C}$ 
with $(a_{N, 1}, \ldots, a_{N, k'_{N}} ) \neq (0,\ldots, 0)$
such that
\begin{align}
P_{\varLambda_N}(S_{p, N}^{\mu}) \phi  = a_{N,1} \phi_{N,1} + \cdots + a_{N, k'_{N}} \phi_{N, k'_{N}}
\label{eq:lin_comb_eigenvec}
\end{align}
for $N > N_{0}$. 
Here  we define $i_{N}(1)$ as a minimizer of $| \lambda - \lambda_{N, i} |$ for $i$ 
with $a_{N,i}\neq 0$ in \eqref{eq:lin_comb_eigenvec}, i.e.
\begin{align}
| \lambda - \lambda_{N, i_{N}(1)} | 
= \min \{ | \lambda - \lambda_{N, i} | \mid i: a_{N,i}\neq 0 \text{ in \eqref{eq:lin_comb_eigenvec}} \}.
\label{eq:def_min_i}
\end{align}
Moreover, we define $\{ i_{N}(2),\ldots, i_{N}(k'_{N}) \}$ as a permutation of $\{1,\ldots, k'_{N}\} \setminus \{ i_{N}(1) \}$. 
Consequently, we obtain the ordered set $\{ \lambda_{N, i_{N}(1)}, \ldots,  \lambda_{N, i_{N}(k'_{N})} \}$ for $N > N_{0}$. 
For simplicity, 
let $\lambda_{N, j}$, $\phi_{N, j}$ and $a_{N, j}$ denote 
$\lambda_{N, i_{N}(j)}$, $\phi_{N, i_{N}(j)}$ and $a_{N, i_{N}(j)}$ for $j=1,\ldots, k'_{N}$, 
respectively. 
In the following we consider \eqref{eq:lin_comb_eigenvec} under this renumbering. 

Here we begin the estimate of the convergence rate of $\lambda_{N,1}$. 
Applying the operator $(S_{p, N}^{\mu} - \lambda_{N,2} I)\cdots (S_{p, N}^{\mu} - \lambda_{N,k'_{N}} I)$ to the both side of \eqref{eq:lin_comb_eigenvec},
we have
\begin{align*}
g_{N}(S_{p, N}^{\mu})\, P_{\varLambda_N}(S_{p, N}^{\mu})\, \phi = a_{N, 1}\, g_{N}(\lambda_{N,1}) \, \phi_{N,1}, 
\end{align*}
where $g_{N}$ is the polynomial defined as
\begin{align}
g_{N}(x) = (x - \lambda_{N,2})\cdots (x - \lambda_{N,k'_{N}}).
\label{eq:tmp_g_N}
\end{align}
Because $a_{N, 1} \neq 0$ and $\lambda_{N,1} \neq \lambda_{N,i}\ (i = 2,\ldots , k'_{N})$, 
we can define a normalized eigenvector $\phi_{N}$ corresponding to $\lambda_{N,1}$ as 
\begin{align}
\phi_{N} = C_{N}^{-1}\, g_{N}(S_{p, N}^{\mu})\, P_{\varLambda_N}(S_{p, N}^{\mu})\, \phi,
\label{eq:def_phi_N_plural}
\end{align}
where
\begin{align}
& C_{N} = \mathrm{sign}(g_{N}(\lambda))\, \|\, g_{N}(S_{p, N}^{\mu})\, P_{\varLambda_N}(S_{p, N}^{\mu})\, \phi \|_{2},
\label{eq:normalizing_C} 
\end{align}
Using $\phi_{N}$ in \eqref{eq:def_phi_N_plural},
in the same manner as \eqref{eq:LamdaEstim1}, 
we can derive an estimate:
\begin{align}
|\lambda - \lambda_{N,1}|
\leq (|\lambda| + |\lambda_{N, 1}| ) \| \phi - \phi_N \|_{2} + \| (S_{p}^{\mu} - S_{p, N}^{\mu}) \phi \|_{2}. 
\label{eq:LamdaEstim1_plural}
\end{align}
Multiplying the both side of \eqref{eq:LamdaEstim1_plural} by
$| g_{N}(\lambda) |$, 
we have
\begin{align}
\prod_{i=1}^{k'_{N}} | \lambda - \lambda_{N, i} | 
\leq
(|\lambda| +  |\lambda_{N, 1}|  )\, | g_{N}(\lambda) |\, \| \phi - \phi_N \|_{2} + 
| g_{N}(\lambda) |\, \| (S_{p}^{\mu} - S_{p, N}^{\mu}) \phi \|_{2}. 
\label{eq:LamdaEstim1.5_plural}
\end{align}
Noting Assumption~\ref{assump:approx_cond_plural} and \eqref{eq:def_min_i}, 
we can deduce from \eqref{eq:LamdaEstim1.5_plural} that 
\begin{align}
| \lambda - \lambda_{N, 1} |^{k'_{N}}
\leq
(|\lambda| +  |\lambda_{N, 1}|  )\, | g_{N}(\lambda) |\, \| \phi - \phi_N \|_{2} + 
(r_{\lambda}/2)^{k'_{N}-1} \| (S_{p}^{\mu} - S_{p, N}^{\mu}) \phi \|_{2}. 
\label{eq:LamdaEstim2_plural}
\end{align}
Then what remains is to estimate $ | g_{N}(\lambda) |\, \| \phi - \phi_N \|_{2}$. 
By the triangle inequality, we have
\begin{align*}
| g_{N}(\lambda) |\, \| \phi - \phi_N \|_{2} 
& \leq 
\| g_{N}(\lambda) \phi - C_{N} \phi_{N} \|_{2} 
+
\left \| \left( C_{N} - g_{N}(\lambda) \right) \phi_N \right \|_{2}, 
\end{align*}
and
\begin{align}
& \| g_{N}(\lambda) \phi - C_{N} \phi_{N} \|_{2} \notag \\
& = 
\| 
g_{N}(S_{p}^{\mu})\, P_{\lambda} (S_{p}^{\mu})\, \phi - g_{N}(S_{p, N}^{\mu})\, P_{\varLambda_{N}} (S_{p, N}^{\mu})\, \phi  
\|_{2} \notag \\
& \leq 
\| 
\{ g_{N}(S_{p}^{\mu}) - g_{N}(S_{p, N}^{\mu}) \}\, \phi
\|_{2} 
+
\|
g_{N}(S_{p, N}^{\mu}) \{ P_{\lambda}(S_{p}^{\mu}) - P_{\varLambda_{N}}(S_{p, N}^{\mu}) \}\, \phi
\|_{2} .
\label{eq:LamdaEstim3_plural}
\end{align}
Therefore setting 
\begin{align*}
& E_{1, N}
= \| \left( C_{N} - g_{N}(\lambda) \right) \phi_N \|_{2}, \\
& E_{2, N} 
= \| \{ g_{N}(S_{p}^{\mu}) - g_{N}(S_{p, N}^{\mu}) \}\, \phi \|_{2}, \\
& E_{3, N} 
= \| g_{N}(S_{p, N}^{\mu}) \{ P_{\lambda}(S_{p}^{\mu}) - P_{\varLambda_{N}}(S_{p, N}^{\mu}) \}\, \phi \|_{2}
\end{align*}
we have
\begin{align}
| g_{N}(\lambda) |\, \| \phi - \phi_N \|_{2} 
\leq E_{1, N} + E_{2, N} + E_{3, N}. 
\label{eq:main_split_1}
\end{align}
We can estimate $E_{1, N}$, $E_{2, N}$ and $E_{3, N}$ as follows. 
Using \eqref{eq:def_phi_N_plural} and \eqref{eq:normalizing_C}, we have
\begin{align}
E_{1, N} 
& = |C_{N} - g_{N}(\lambda) | 
   = |C_{N} - \mathrm{sign}(g_{N}(\lambda)) \| g_{N}(\lambda) \phi \|_{2} | \notag \\
& = | \| C_{N} \phi_{N} \|_{2} - \| g_{N}(\lambda) \phi \|_{2} | 
   \leq \| C_{N} \phi_{N} - g_{N}(\lambda) \phi \|_{2} \notag \\
& \leq E_{2, N} + E_{3, N}, 
\label{eq:EstimE1}
\end{align}
where the last inequality is due to \eqref{eq:LamdaEstim3_plural}. 
Noting that $\phi$ is an eigenvector of $S_{p}^{\mu}$ corresponding to $\lambda$, we have
\begin{align}
\{ g_{N}(S_{p}^{\mu}) - g_{N}(S_{p, N}^{\mu}) \} \phi 
& = \{ g_{N}(\lambda) - g_{N}(S_{p, N}^{\mu}) \} \phi \notag \\
& = \sum_{i=0}^{k'_{N}-1} 
\alpha(i, \Lambda_{N})
\left( \lambda^{k'_{N} - 1 - i} I  - (S_{p, N}^{\mu})^{k'_{N} - 1 - i} \right)\, \phi \notag \\
& = \sum_{i=0}^{k'_{N}-2}
\alpha(i, \Lambda_{N})\, 
\beta(i, \lambda, S_{p, N}^{\mu})\,  (\lambda I - S_{p, N}^{\mu})\, \phi \notag \\
& = \sum_{i=0}^{k'_{N}-2} 
\alpha(i, \Lambda_{N})\, 
\beta(i, \lambda, S_{p, N}^{\mu})\, (S_{p}^{\mu} - S_{p, N}^{\mu})\, \phi, 
\label{eq:EstimE2_pre}
\end{align}
where
\(
 \alpha(i, \Lambda_{N})= (-1)^{i} \sum_{2\leq j_{1} < \cdots < j_{i} \leq k'_{N}} \lambda_{N, j_{1}}\cdots \lambda_{N, j_{i}} 
\)
and 
\(
 \beta(i, \lambda, S_{p, N}^{\mu}) = 
 \sum_{l = 0}^{k'_{N} - 2 - i} \lambda^{k'_{N} - 2 - i - l} (S_{p, N}^{\mu})^{l}
\). 
Therefore noting 
\begin{align*}
| \alpha(i, \Lambda_{N}) |
& \leq \left(|\lambda| + \frac{r_{\lambda}}{2}\right)^{i}
\leq \max \left\{ 1, \left(|\lambda| + \frac{r_{\lambda}}{2}\right)^{k'_{N}} \right\}, \\
\|
\beta(i, \lambda, S_{p, N}^{\mu})
\|_{2}
& \leq 
\max \left\{1, |\lambda|^{k'_{N}} \right\} 
\sum_{l = 0}^{k'_{N} - 2 - i} \|S_{p, N}^{\mu}\|_{2}^{l} \\
& \leq 
\max \left\{1, |\lambda|^{k'_{N}} \right\} \, 
k'_{N} \max \left\{1, \|S_{p, N}^{\mu}\|_{2}^{k'_{N}} \right\}
\end{align*}
and \eqref{eq:EstimE2_pre}, we have
\begin{align}
E_{2, N}
\leq 
(k'_{N})^2
\max \left\{ 1, \left(|\lambda| + \frac{r_{\lambda}}{2}\right)^{2k'_{N}} \right\} 
\max \left\{1, \|S_{p, N}^{\mu}\|_{2}^{k'_{N}} \right\}
\| (S_{p}^{\mu} - S_{p, N}^{\mu}) \phi \|_{2}.
\label{eq:EstimE2}
\end{align}
Using \eqref{eq:EstimByResol_2_2b} in Lemma~\ref{lem:make_approx_eigenvec_2b}, 
we have 
\begin{align}
E_{3, N} 
& \leq 
\left( \prod_{i = 2}^{k'_{N}} \| S_{p, N}^{\mu} - \lambda_{N, i} I  \|_{2} \right)
\| \{ P_{\lambda}(S_{p}^{\mu}) - P_{\varLambda_{N}}(S_{p, N}^{\mu}) \} \phi \|_{2} \notag \\
& \leq 
\left( \prod_{i = 2}^{k'_{N}} ( \| S_{p, N}^{\mu} \|_{2} + | \lambda_{N, i} | ) \right)
\| \phi - P_{\varLambda_{N}}(S_{p, N}^{\mu}) \phi \|_{2} \notag \\
& \leq 
\left( \| S_{p, N}^{\mu} \|_{2} + | \lambda | + \frac{r_{\lambda}}{2} \right)^{k'_{N} - 1}
\frac{2}{r_{\lambda}} \| (S_{p}^{\mu} - S_{p, N}^{\mu}) \phi \|_{2}. 
\label{eq:EstimE3}
\end{align}

Finally, noting $k'_{N} \leq k$ and 
combining 
\eqref{eq:LamdaEstim2_plural}, 
\eqref{eq:main_split_1}, 
\eqref{eq:EstimE1},
\eqref{eq:EstimE2} and
\eqref{eq:EstimE3}, 
we have 
\begin{align*}
& | \lambda - \lambda_{N, 1} |^{k'_{N}} \\
& \leq
2 (|\lambda| + | \lambda_{N,1} |) ( E_{2,N} + E_{3,N} ) + 
(r_{\lambda}/2)^{k'_{N}-1} \| (S_{p}^{\mu} - S_{p, N}^{\mu}) \phi \|_{2}. \\
& \leq 
\left[
(4 |\lambda| + r_{\lambda} ) 
\left[
k^2
\max \left\{ 1, \left(|\lambda| + \frac{r_{\lambda}}{2}\right)^{2k} \right\} 
\max \left\{1, \|S_{p, N}^{\mu}\|_{2}^{k} \right\} 
 \right. 
 \right. \\
& \left. 
 \left.
\quad +\frac{2}{r_{\lambda}} 
\max \left\{ 1, \left( \| S_{p, N}^{\mu} \|_{2} + | \lambda | + \frac{r_{\lambda}}{2} \right)^{k} \right\}
\right]
+ \max \left\{1, \left( \frac{r_{\lambda}}{2} \right)^{k} \right\}
\right] 
\| (S_{p}^{\mu} - S_{p, N}^{\mu}) \phi \|_{2}.
\end{align*}
Thus we obtain the conclusion.
\end{proof}

Here we prove Theorem~\ref{thm:main03}. 

\renewcommand{\proofname}{Proof of Theorem~\ref{thm:main03}}
\begin{proof}
In a similar manner to the proof of Theorem~\ref{thm:main01}, 
we can deduce the conclusion~\eqref{eq:main_03}
from 
Lemmas~\ref{lem:norm_P_N_S}, \ref{lem:ProjEstim_C_infty} and~\ref{eq:StrResolEstim_02}, 
and~\eqref{eq:main_04}
from
Lemmas~\ref{lem:norm_P_N_S}, \ref{lem:ProjEstim_analytic}, and~\ref{eq:StrResolEstim_02}. 
Note that $\| S_{p, N}^{\mu} \|_{2} \leq \| \hat{P}_{N} S_{p}^{\mu} \|_{2} \| \hat{P}_{N} \|_{2} = \| \hat{P}_{N} S_{p}^{\mu} \|_{2}$. 
\end{proof}
\renewcommand{\proofname}{Proof}

\subsection{Proof of Theorem~\ref{thm:local_2b_4a}}
\label{subsec:local}

To prove the theorem, 
we use a similar technique to Lemmas~\ref{lem:make_approx_eigenvec_2a} and~\ref{eq:StrResolEstim}, 
in which we exchange the roles of $\phi$ and $\phi_{N}$. 
That is, we take an eigenvector $\phi_{N}$ with $\| \phi_{N} \|_{2} = 1$ corresponding to $\lambda_{N}$, 
and using $\phi_{N}$ and the projection $P_{\lambda}(S_{p}^{\mu})$ for $\lambda \in \sigma(S_{p}^{\mu})$, 
we obtain an eigenvector $\phi$ with $\| \phi \|_{2} = 1$ corresponding to $\lambda$. 
First, we present a lemma guaranteeing the feasibility of such procedure, 
which corresponds to Lemma~\ref{lem:make_approx_eigenvec_2a}.
We omit its proof since it is proved in almost the same manner as Lemma~\ref{lem:make_approx_eigenvec_2a}.

\begin{lem}
\label{lem:make_approx_eigenvec_reverse}
Let $S_{p}^{\mu}$ satisfy 
Assumptions~\ref{assump:self-adjoint} and~\ref{assump:C_infty},  
$\lambda \in \sigma(S_{p}^{\mu})$ and $\lambda_{N} \in \sigma(S_{p, N}^{\mu})$
satisfy Assumption~\ref{assump:approx_cond_single}. 
Then for an eigenvector $\phi_{N}$ corresponding to $\lambda_{N}$ we have
\begin{align}
\| \phi_{N} - P_{\lambda}(S_{p}^{\mu}) \phi_{N} \|_{2}  
\leq 
\frac{2}{r_{\lambda}} \| (S_{p}^{\mu} - S_{p, N}^{\mu}) \phi_{N} \|_{2}.
\label{eq:EstimByResol_2_reverse}
\end{align} 
Furthermore, if $\| \phi_{N} \|_{2} = 1$ and 
\begin{align}
\frac{2}{r_{\lambda}} \| (S_{p}^{\mu} - S_{p, N}^{\mu}) \phi_{N} \|_{2} < 1, 
\label{eq:less_than_one_reverse}
\end{align}
we have $P_{\lambda}(S_{p}^{\mu}) \phi_{N} \neq 0$.
\end{lem}

Next, we present a counterpart of Lemma~\ref{eq:StrResolEstim}, 
whose proof is also omitted. 

\begin{lem}\label{eq:StrResolEstim_reverse}
Let $S_{p}^{\mu}$ satisfy 
Assumptions~\ref{assump:self-adjoint} and~\ref{assump:C_infty},  
$\lambda \in \sigma(S_{p}^{\mu})$ and $\lambda_{N} \in \sigma(S_{p, N}^{\mu})$
satisfy Assumption~\ref{assump:approx_cond_single}. 
Furthermore, for an eigenvector $\phi_{N}$ 
corresponding to $\lambda_{N}$ 
with $\| \phi_{N} \|_{2} = 1$, assume that 
\begin{align}
\frac{2}{r_{\lambda}} \| (S_{p}^{\mu} - S_{p, N}^{\mu}) \phi_{N} \|_{2} < 1.
\label{eq:A-ANphi_less_than_1_reverse}
\end{align}
Then we have
\begin{align}
| \lambda - \lambda_{N} | 
\leq
\left(
3 + \frac{8 |\lambda|}{r_{\lambda}}
\right)
\| (S_{p}^{\mu} - S_{p, N}^{\mu}) \phi_{N} \|_{2}. 
\label{eq:GeneralRate_reverse}
\end{align} 
\end{lem}

Here we prove Theorem~\ref{thm:local_2b_4a}.

\renewcommand{\proofname}{Proof of Theorem~\ref{thm:local_2b_4a}}
\begin{proof}
It follows from the assumption~\eqref{eq:local_2b}, 
Theorems~\ref{thm:no-spurious-mode} and~\ref{thm:all_lambda_produced}
that the set of the accumulation points of the sequence $\{ \lambda_{N} \}$ 
coincides with $\sigma(S_{p}^{\mu}) \cap B_{\zeta}(r)$. 
Moreover, 
the elements of $\sigma(S_{p}^{\mu}) \setminus B_{\zeta}(r)$ 
do not exist in the interior of $B_{\zeta}(9r)$. 
In the following we set $\varLambda_{\zeta, r} = \sigma(S_{p}^{\mu}) \cap B_{\zeta}(r)$. 

First, we show that $\{ \lambda_{N} \}$ is convergent and 
$\lambda \in \varLambda_{\zeta, r}$ 
is uniquely determined as the limit of $\{ \lambda_{N} \}$. 
To prove $\{ \lambda_{N} \}$ is Cauchy, 
we choose an arbitrary integer $M>0$ and estimate 
$| \lambda_{N+M} - \lambda_{N} |$ for sufficiently large $N$. 
For an eigenvalue $\lambda \in \varLambda_{\zeta, r}$, 
we take an eigenvector $\phi$ with $\| \phi \|_{2} = 1$ corresponding to $\lambda$. 
Then $\phi$ satisfies
\begin{align}
P_{\varLambda_{\zeta, r}}(S_{p}^{\mu}) \phi 
= \frac{1}{2\pi \mathrm{i}} \oint_{\varGamma_{\zeta, r}} (\nu I - S_{p}^{\mu})^{-1} \phi\, \mathrm{d}\nu
= \frac{1}{2\pi \mathrm{i}} \oint_{\varGamma_{\zeta, r}} (\nu - \lambda)^{-1} \phi\, \mathrm{d}\nu 
= \phi,
\label{eq:projection}
\end{align}
where $\varGamma_{\zeta, r}$ is the boundary of $B_{\zeta}(2r)$ with counterclockwise direction.
Note that \eqref{eq:projection} holds true regardless of the number of the elements in 
$\varLambda_{\zeta, r}$. 
Using this $\phi$, we define eigenvectors corresponding to $\lambda_{N+M}$ and $\lambda_{N}$ as
\begin{align*}
& \phi_{N+M} = \frac{1}{ \| P_{\lambda_{N+M}}(S_{p, N+M}^{\mu}) \phi  \|_{2} } P_{\lambda_{N+M}}(S_{p, N+M}^{\mu}) \phi \\
\intertext{and}
& \phi_{N} = \frac{1}{ \| P_{\lambda_{N}}(S_{p, N}^{\mu}) \phi  \|_{2} } P_{\lambda_{N}}(S_{p, N}^{\mu}) \phi, 
\end{align*}
respectively. Here we use again $\varGamma_{\zeta, r}$ to define 
$P_{\lambda_{N+M}}(S_{p, N+M}^{\mu})$ and $P_{\lambda_{N}}(S_{p, N}^{\mu})$. 
The availability of the eigenvectors can be guaranteed by a similar manner to Lemma~\ref{lem:make_approx_eigenvec_2a}. 
Then noting $S_{p, N}^{\mu} = \hat{P}_{N} S_{p}^{\mu} \hat{P}_{N}$ etc.~we have
\begin{align*}
\lambda_{N+M} - \lambda_{N} 
& = (S_{p, N+M}^{\mu} \phi_{N+M}, \phi_{N+M}) - (S_{p, N+M}^{\mu} \phi_{N}, \phi_{N}) \\
& = (S_{p, N+M}^{\mu} \phi_{N+M}, \phi_{N+M} - \phi_{N} ) + (S_{p, N+M}^{\mu} (\phi_{N+M} - \phi_{N}), \phi_{N}) \\
& = (S_{p, N+M}^{\mu} \phi_{N+M}, \phi_{N+M} - \phi_{N} ) + ( \phi_{N+M} - \phi_{N}, S_{p, N+M}^{\mu} \phi_{N}) 
\end{align*}
and therefore
\begin{align}
| \lambda_{N+M} - \lambda_{N} |
& \leq 2 \| S_{p, N+M}^{\mu} \|_{2} (\| \phi_{N+M} - \phi \|_{2}  + \| \phi - \phi_{N} \|_{2} ). 
\label{eq:Cauchy_1}
\end{align}
Using a similar manner to Lemma~\ref{eq:StrResolEstim}, 
Lemma~\ref{lem:norm_P_N_S} and 
Lemma~\ref{lem:ProjEstim_C_infty}, 
we can show that the RHS of \eqref{eq:Cauchy_1} is 
$\mathrm{O}(N^{-q})$ for any positive integer $q$. 
In fact, $\| S_{p, N+M}^{\mu} \|_{2} = \mathrm{O}(N^{p+1/2})$ by Lemma~\ref{lem:norm_P_N_S} and 
\begin{align*}
\| \phi - \phi_{N} \|_{2} 
& \leq 2\, \| \phi - P_{\lambda_{N}}(S_{p, N}^{\mu}) \phi \|_{2}
= 2\, \| P_{\varLambda_{\zeta, r}}(S_{p}^{\mu}) \phi - P_{\lambda_{N}}(S_{p, N}^{\mu}) \phi \|_{2} \\
& \leq 4r \max_{\nu \in \varGamma_{\zeta, r}} 
\| (\nu I - S_{p, N}^{\mu})^{-1} (S_{p, N}^{\mu} - S_{p}^{\mu}) (\nu I - S_{p}^{\mu})^{-1} \phi \|_{2}\\
& = 4r \max_{\nu \in \varGamma_{\zeta, r}} 
\| (\nu I - S_{p, N}^{\mu})^{-1} (S_{p, N}^{\mu} - S_{p}^{\mu}) (\nu - \lambda)^{-1} \phi \|_{2} \\
& \leq 4r \max_{\nu \in \varGamma_{\zeta, r}} 
\frac{1}{| \nu - \lambda |} \frac{1}{| \nu - \lambda_{N} |} \| (S_{p, N}^{\mu} - S_{p}^{\mu}) \phi \|_{2} \\
& \leq \frac{4}{r} \| (S_{p, N}^{\mu} - S_{p}^{\mu}) \phi \|_{2}.
\end{align*}
Thus we can show that $| \lambda_{N+M} - \lambda_{N} | \to 0$ as $N \to \infty$ 
and $\lambda \in \varLambda_{\zeta, r}$ is uniquely determined. 
Combining this fact and $|\lambda - \lambda_{N}| \leq 2r$, 
we can show that Assumption~\ref{assump:approx_cond_single}
is satisfied for $\{ \lambda_{N} \}$, $\lambda$, and $r_{\lambda} = 4r$. 
In fact, for any 
$\tilde{\lambda} \in \sigma(S_{p}^{\mu}) \cup \sigma(S_{p, N}^{\mu}) \setminus \{ \lambda, \lambda_{N} \}$
we have
\[
| \tilde{\lambda} - \lambda | 
\geq 
| \tilde{\lambda} - \zeta | - | \lambda - \zeta |
\geq 
2r_{\lambda}. 
\]

Next, we show the error bound~\eqref{eq:bound}. 
By Lemma~\ref{eq:StrResolEstim_reverse}, 
it follows from \eqref{eq:GeneralRate_reverse} and $| \lambda | \leq r + |\zeta|$ that 
\begin{align}
| \lambda - \lambda_{N} | 
\leq
\left(
5 + \frac{3 |\zeta|}{r}
\right)
\| (S_{p}^{\mu} - S_{p, N}^{\mu}) \phi_{N} \|_{2}
\label{eq:pre_bound_1}
\end{align}
for $\lambda \in \sigma(S_{p}^{\mu}) \cap B_{\zeta}(r)$.
Furthermore, we have 
\begin{align*}
& (S_{p}^{\mu} - S_{p, N}^{\mu}) \phi_{N} 
=
( I - \hat{P}_{N}) S_{p}^{\mu} \phi_{N} \notag \\
& = 
\sum_{|l| > N} 
\left\{
\sum_{j=0}^{p} \sum_{m = l - N}^{l + N} 
(\hat{f}_{j})_{m} (\hat{\phi}_{N})_{l-m} \frac{1}{\sqrt{L}} \left( -\mathrm{i}\frac{2\pi (l-m)}{L} \right)^{j}
\right\}
\frac{1}{\sqrt{L}} \exp\left( -\mathrm{i} \frac{2\pi l}{L} x \right), 
\end{align*}
and therefore
\begin{align}
\| (S_{p}^{\mu} - S_{p, N}^{\mu}) \phi_{N} \|_{2}
& \leq
\sum_{|l| > N} 
\left|
\sum_{j=0}^{p} \sum_{m = l - N}^{l + N} 
(\hat{f}_{j})_{m} (\hat{\phi}_{N})_{l-m} \frac{1}{\sqrt{L}} \left( -\mathrm{i}\frac{2\pi (l-m)}{L} \right)^{j}
\right| \notag \\
& \leq 
\frac{1}{\sqrt{L}} \left( \frac{2\pi N}{L} \right)^{p}
\sum_{j=0}^{p} 
\sum_{|l| > N} 
\sum_{m = l - N}^{l + N} 
\left| (\hat{f}_{j})_{m} (\hat{\phi}_{N})_{l-m} \right|.
\label{eq:pre_bound_2}
\end{align}
Considering the case $|l| \geq 2N$, we have
\begin{align}
\sum_{|l| \geq 2N} \sum_{m = l - N}^{l + N} 
\left| (\hat{f}_{j})_{m} (\hat{\phi}_{N})_{l-m} \right|
\leq 
\sum_{|l| \geq 2N} \sum_{m = l - N}^{l + N} 
\left| (\hat{f}_{j})_{m} \right|
\leq 
(2N+1)
\sum_{|m| \geq N} 
\left| (\hat{f}_{j})_{m} \right|.
\label{eq:pre_bound_3}
\end{align}
Combining~\eqref{eq:pre_bound_1}, \eqref{eq:pre_bound_2} and~\eqref{eq:pre_bound_3}, 
we obtain the error bound~\eqref{eq:bound}. 
\end{proof}
\renewcommand{\proofname}{Proof}

\section{Concluding Remarks}
\label{sec:ConcRem}

We have considered the convergence rates and the explicit error bounds
of Hill's method, which is a numerical method for computing the spectra of  
self-adjoint differential operators $S_{p}$ in~\eqref{eq:main_prob} with periodic coefficient functions $\tilde{f}_{j}$. 
On the assumption (Assumptions~\ref{assump:approx_cond_single} and~\ref{assump:approx_cond_plural})
that the computed eigenvalue $\lambda_{N}$ is  close to the exact one $\lambda$, 
it is shown in Theorems~\ref{thm:main01} and~\ref{thm:main03} 
that the convergence rate of the computed eigenvalue is all order polynomial
in the case of $\tilde{f}_{j}$ being $C^{\infty}$ as shown in~\eqref{eq:main_01} and~\eqref{eq:main_03}, and exponential
in the case of $\tilde{f}_{j}$ being analytic as shown in~\eqref{eq:main_02} and~\eqref{eq:main_04}. 
In addition, even if the exact eigenvalue $\lambda$ is unknown, 
it is shown in Theorems~\ref{thm:main01}, \ref{thm:main03} and~\ref{thm:local_2b_4a}
that, if the condition \eqref{eq:local_2b} is satisfied, 
we can obtain the convergence rate using \eqref{eq:main_01}--\eqref{eq:main_04}
and the explicit error bound using \eqref{eq:bound}.
There are some cases in which the condition \eqref{eq:local_2b} can be checked using 
Gershgorin's theorem. 
These theorems are proved using the Dunford integrals $P_{\lambda}(S_{p, N}^{\mu})$ in \eqref{eq:DunfordIntP} 
and $P_{\lambda}(S_{p}^{\mu})$ in the proof of Theorem~\ref{thm:local_2b_4a}, 
which project an eigenvector to the corresponding eigenspace. 
This integral is suitable for proving these theorems, 
because the boundary of the neighborhood of an eigenvalue can be directly 
used as the contour of the integral. 
Numerical examples using Hill's operator \eqref{eq:HillEqSn} support these theoretical results. 



As described in Remark \ref{rem:SelfAd}, 
self-adjointness of the operator $S_{p}^{\mu}$ enables us to estimate 
the norm of the resolvent as shown in \eqref{eq:ResolCalc}, 
and to project an eigenvector to the corresponding eigenspace using the Dunford integral.
For the case of non-self-adjoint operators, 
we have developed some other approaches, 
which will be reported somewhere else soon. 


\section*{Acknowledgments}

This work is partially supported by Grant-in-Aid for Scientific Research (S) 20224001.

\appendix 

\section{Proofs of Lemmas~\ref{lem:norm_P_N_S}, \ref{lem:ProjEstim_C_infty}, and~\ref{lem:ProjEstim_analytic}}
\label{sec:ProofOrders}

First, we prove Lemma~\ref{lem:norm_P_N_S}.

\renewcommand{\proofname}{Proof of Lemma~\ref{lem:norm_P_N_S}}
\begin{proof}
We set $f_{p} \equiv 1$ for simplicity. 
Let $\mathrm{e}_{m}\ (m=0, \pm 1, \pm 2, \ldots)$ be the Fourier bases defined as
\[
\mathrm{e}_{m}(x) = \frac{1}{\sqrt{L}} \exp \left( -\mathrm{i} \frac{2\pi m}{L} x \right), 
\]
and let $(\hat{f}_j)_{m}\ (j=0,1,\ldots, p,\ m=0, \pm 1, \pm 2, \ldots)$ 
denote the Fourier coefficients of $f_j\ (j=0,1,\ldots, p)$, i.e.~$(\hat{f}_j)_{m} = (f_{j},\, \mathrm{e}_{m})_{2}$. 
Then it is a standard fact that 
under Assumption~\ref{assump:C_infty}, 
for $j = 0,1,\ldots, p$, $m=0, \pm 1, \pm 2, \ldots$ and any positive integer $q$, 
there exists a positive real constant $C_{j, q}$ 
depending only on $f_{j}$ and $q$
such that 
\begin{align}
|(\hat{f}_j)_{m}| \leq C_{j, q}\, ( 1 + |m| )^{-q}.
\label{eq:C_infty_Fourier_decay} 
\end{align}

We estimate the norm of 
\begin{align}
(\hat{P}_N S_{p}^{\mu} \psi)(x) = 
\sum_{n = -N}^{N} 
\sum_{j=0}^{p}
\left\{ \sum_{m=-\infty}^{\infty} 
(\hat{f}_j)_{m} \left( -\mathrm{i} \frac{2\pi (n-m) }{L} \right)^{j} \hat{\psi}_{n-m}  
\right\} \mathrm{e}_n(x)
\label{eq:PApsi_Fourier}
\end{align}
for $N\geq 1$ and a nonzero vector $\psi \in L_2([0,L])_{\mathrm{per}}$. 
Setting 
\begin{align*}
c_j(n,m) =  (\hat{f}_j)_{m} \left( -\mathrm{i} \frac{2\pi (n-m) }{L} \right)^{j}
\end{align*}
and $q=p+1$ in \eqref{eq:C_infty_Fourier_decay}, 
for $j = 0,1,\ldots, p$, $m=0, \pm 1, \pm 2, \ldots$ and $n$ with $|n| \leq N$, we have
\begin{align*}
N^{-j} | c_j(n,m) |
& = \left( \frac{2\pi }{L} \right)^{j} | (\hat{f}_j)_{m} |  \left( \frac{|n| + |m|}{N} \right)^{j} \notag \\
& \leq \left( \frac{2\pi }{L} \right)^{j} C_{j, p+1}\, ( 1 + |m| ) ^{-p-1} \left( 1 + \frac{|m|}{N} \right)^{j} \notag \\
& \leq \left( \frac{2\pi }{L} \right)^{j} C_{j, p+1}\, ( 1 + |m| )^{j - p - 1}, 
\end{align*}
and therefore
\begin{align}
 | c_j(n,m) | \leq \left( \frac{2\pi }{L} \right)^{j} \frac{C_{j, p+1}}{1 + |m|}\, N^{j}.
 \label{eq:Fourier_coeff_estim}
\end{align}
Combining \eqref{eq:PApsi_Fourier} and \eqref{eq:Fourier_coeff_estim} we have
\begin{align}
\| \hat{P}_N S_{p}^{\mu} \psi \|_2
& \leq \sum_{j=0}^{p} 
\left\| \sum_{n=-N}^{N} \left( \sum_{m=-\infty}^{\infty} c_j(n,m) \hat{\psi}_{n-m} \right) \mathrm{e}_n \right\|_2 \notag \\
& = \sum_{j=0}^{p} \left\{ \sum_{n=-N}^{N} \left| \sum_{m=-\infty}^{\infty} c_j(n,m) \hat{\psi}_{n-m} \right|^2 \right\}^{1/2} \notag \\
& \leq \sum_{j=0}^{p} 
\left\{ 
\sum_{n=-N}^{N} \left( \sum_{m=-\infty}^{\infty} |c_j(n,m)|^2 \right) \left( \sum_{m=-\infty}^{\infty} |\hat{\psi}_{n-m}|^2 \right) 
\right\}^{1/2} \notag \\
& \leq \sum_{j=0}^{p} 
\left\{ 
\sum_{n=-N}^{N} \left( \frac{2\pi }{L} \right)^{2j} C_{j, p+1}^2 N^{2j} 
\left( \sum_{m=-\infty}^{\infty} \frac{1}{(1+|m|)^2} \right) \|\psi\|_2^2
\right\}^{1/2} \notag \\
& \leq  \sum_{j=0}^{p} \tilde{C}_{j, p+1} (2N+1)^{1/2} N^j \|\psi\|_2, 
\label{eq:PNApsi_estim}
\end{align}
where $ \tilde{C}_{j, p+1}\ (j=0,1,\ldots, p)$ are bounded constants defined as
\[
 \tilde{C}_{j, p+1} = 
 \left( \frac{2\pi }{L} \right)^{j} C_{j, p+1}
 \left( \sum_{m=-\infty}^{\infty} \frac{1}{(1+|m|)^2} \right)^{1/2}.
\]
Hence it follows from \eqref{eq:PNApsi_estim} that
\begin{align}
\| \hat{P}_N S_{p}^{\mu}  \|_2 \leq \sum_{j=0}^{p} \tilde{C}_{j, p+1} (2N+1)^{1/2} N^j.
\label{eq:PNA_estim}
\end{align}
Thus we have the conclusion.
\end{proof}
\renewcommand{\proofname}{Proof}

Next, we prove Lemma~\ref{lem:ProjEstim_C_infty} and~\ref{lem:ProjEstim_analytic}
with the aid of the following two lemmas. 
Since these lemmas are fundamental in the theory of differential equations, 
we only give the outlines of their proofs in Appendix~\ref{sec:ProofEigenSmooth}.

\begin{lem}
\label{lem:EigenSmooth_C_infty}
Let $S_{p}^{\mu}$ satisfy Assumption~\ref{assump:C_infty} and 
$\phi \in H^{p}([0,L])_{\mathrm{per}}$ be an eigenvector of $S_{p}^{\mu}$ corresponding to 
$\lambda \in \sigma(S_{p}^{\mu} )$.
Then $\phi$ is $C^{\infty}$. 
\end{lem}

\begin{lem}
\label{lem:EigenSmooth_analytic}
Let $S_{p}^{\mu}$ satisfy Assumption~\ref{assump:analytic} for $d>0$ and 
$\phi \in H^{p}([0,L])_{\mathrm{per}}$ be an eigenvector of $S_{p}^{\mu}$ 
corresponding to $\lambda \in \sigma(S_{p}^{\mu} )$.
Then $\phi$ can be uniquely extended to an analytic function on $\mathcal{D}_{d}$
with $\phi(z + L) = \phi(z)$ for any $z \in \mathbf{C}$, 
where $\mathcal{D}_{d}$ is defined as \eqref{eq:DefD_d} and 
$L$ is the period of the coefficient functions of $S_{p}^{\mu}$. 
\end{lem}

\renewcommand{\proofname}{Proof of Lemma~\ref{lem:ProjEstim_C_infty}}
\begin{proof}
Since $\phi$ is $C^{\infty}$ by Lemma~\ref{lem:EigenSmooth_C_infty}, 
in a similar manner to \eqref{eq:C_infty_Fourier_decay}, 
we can use the fundamental fact for the Fourier coefficients that 
there exists a positive constant $C_{\phi, q}$ depending only on $\phi$ and $q$ such that
\[
|\hat{\phi}_{n}| \leq C_{\phi, q}\, (1 + |n|)^{-q-1}
\]
for any integer $n$. Then we have
\begin{align*}
\| (I - \hat{P}_N) \phi \|_{2}^{2} 
& = \sum_{|n| > N} |\hat{\phi}_n|^2 
 \leq \sum_{|n| > N} C_{\phi, q}^2\, (1 + |n|)^{-2q-2} \\
& \leq \frac{C_{\phi, q}^2}{(1+N)^{2q}} \sum_{|n| > N} \frac{1}{(1 + |n|)^{2}}
 \leq \frac{C_{\phi, q}^2}{(1+N)^{2q}} \sum_{n=-\infty}^{\infty} \frac{1}{(1 + |n|)^{2}}.
\end{align*}
Thus we obtain the conclusion.
\end{proof}
\renewcommand{\proofname}{Proof}

\renewcommand{\proofname}{Proof of Lemma~\ref{lem:ProjEstim_analytic}}
\begin{proof}
Since $\phi$ can be regarded as an analytic function on ${\mathcal{D}_{d}}$ with $\phi(z + L) = \phi(z)$
by Lemma~\ref{lem:EigenSmooth_analytic}, there exists a positive constant $C_{\phi, d, \varepsilon}$ 
depending only on $\phi$, $d$ and $\varepsilon$ such that
\begin{align}
 |\hat{\phi}_n| \leq C_{\phi, d, \varepsilon} \exp \left( - \frac{ 2 \pi d _{\varepsilon}}{L} | n | \right)
\label{eq:PW_estim}
\end{align}
for any integer $n$, where $d _{\varepsilon} = d - \varepsilon$. 
In fact, we can show this estimate by the Paley-Wiener type argument.
Noting the periodicity $\phi(z + L) = \phi(z)$, by Cauchy's integral theorem we have 
\begin{align*}
\hat{\phi}_n 
& = \frac{1}{\sqrt{L}} \int_{0}^{L} \phi(x) \exp \left( -\mathrm{i} \frac{2\pi n}{L} x \right)\, \mathrm{d}x \\
& = \frac{1}{\sqrt{L}} 
\int_{0}^{L} \phi \left( x  - \mathrm{sign}(n)\, d_{\varepsilon} \mathrm{i} \right) 
\exp \left( -\mathrm{i} \frac{2\pi n}{L} \left( x  - \mathrm{sign}(n)\, d_{\varepsilon} \mathrm{i} \right) \right)\, \mathrm{d}x \\
& = \exp\left( - \frac{ 2 \pi d _{\varepsilon}}{L} | n | \right) \frac{1}{\sqrt{L}} 
\int_{0}^{L} \phi \left( x  - \mathrm{sign}(n)\, d_{\varepsilon} \mathrm{i} \right) 
\exp \left( -\mathrm{i} \frac{2\pi n}{L} x \right)\, \mathrm{d}x,
\end{align*}
where $\mathrm{sign}(n)$ is the sign of $n$. 
This expression implies \eqref{eq:PW_estim}.
Therefore we have
\begin{align*}
\| (I - \hat{P}_N) \phi \|_{2}^{2} 
 = \sum_{|n| > N} |\hat{\phi}_n|^2 
 \leq \sum_{|n| > N} C_{\phi, d, \varepsilon}^{2} \exp ( - 2\, d_{\varepsilon, L}\, | n | ) 
 = \frac{2\, C_{\phi, d, \varepsilon}^{2} 
\exp ( - 2\, d_{\varepsilon, L}\, (N+1) ) }{ 1 - \exp ( - 2\, d_{\varepsilon, L} )}, 
\end{align*}
where $d_{\varepsilon, L} = 2 \pi d _{\varepsilon} / L$.
Thus we obtain the conclusion.
\end{proof}
\renewcommand{\proofname}{Proof}

\section{Proofs of Lemmas~\ref{lem:EigenSmooth_C_infty} and~\ref{lem:EigenSmooth_analytic}}
\label{sec:ProofEigenSmooth}

\renewcommand{\proofname}{Proof of Lemma~\ref{lem:EigenSmooth_C_infty}}
\begin{proof}
Since $\phi$ is an eigenvector of $S_{p}^{\mu}$ corresponding to $\lambda$, we have
\begin{align*}
\phi^{(p)}(x) = -\sum_{j=0}^{p-1} f_j(x) \phi^{(j)}(x) + \lambda \phi(x). 
\end{align*}
Then it follows from $\phi \in H^{p}([0,L])_{\mathrm{per}}$ that $\phi^{(p)} \in H^{1}([0,L])_{\mathrm{per}}$. 
Hence $\phi \in H^{p+1}([0,L])_{\mathrm{per}}$.
Iterating a similar argument, 
we have $\phi \in H^{q}([0,L])_{\mathrm{per}}$ for any positive integer $q$. 
Then the conclusion follows from the fundamental embedding relation: 
$H^{q}([0,L])_{\mathrm{per}} \hookrightarrow C^{q-1}([0,L])_{\mathrm{per}}$ for any positive integer $q$,
where $C^{q-1}([0,L])_{\mathrm{per}}$ is the set of periodic functions on $[0,L]$ 
which are $q-1$ times continuously differentiable
(see e.g.~\cite[Proposition 7.5.4]{bib:AtkinsonHan_NABook}). 
\end{proof}
\renewcommand{\proofname}{Proof}

\renewcommand{\proofname}{Proof of Lemma~\ref{lem:EigenSmooth_analytic}}
\begin{proof}
We choose an arbitrary $\varepsilon$ with $0 < \varepsilon  < d$ and 
set $d_{\varepsilon} = d - \varepsilon$. 
Then $f_{0}, f_{1},\ldots, f_{p-1}$ are analytic on $\mathcal{D}_{d_{\varepsilon}}$ 
and continuous on $\overline{\mathcal{D}_{d_{\varepsilon}}}$. 
Hence there exists a positive constant $M_{\varepsilon}$ such that 
$| f_{j}(z) | \leq M_{\varepsilon}$ for all $j = 0, 1, \ldots, p-1$ and 
$z \in \mathcal{D}_{d_{\varepsilon}}$. 

By Lemma~\ref{lem:EigenSmooth_C_infty}, we can assume the existence of 
an eigenvector $\phi \in C^{\infty}([0,L])_{\mathrm{per}}$ corresponding to $\lambda \in \sigma(S_{p}^{\mu})$. 
Then for the fixed $\lambda$ and each $x_{0} \in \mathbf{R}$, 
we consider the equation 
$S_{p}^{\mu} \psi = \lambda \psi$ as an initial value problem 
with the condition $\psi^{(n)}(x_{0}) = \phi^{(n)}(x_{0})\ (n=0,\ldots, p-1)$. 
By the matrix $F(z)$ defined as
\begin{align*}
F(z) = 
\begin{pmatrix}
 -f_{p-1}(z) & -f_{p-2}(z) & \cdots & -f_1(z) & -(f_0(z) - \lambda) \\
 1 & 0 & \cdots &   & 0 \\
 0 & 1 & \cdots &   & \vdots  \\
 \vdots & \ddots  & \ddots  &   & \vdots  \\
 0  & \cdots & \cdots & 1 & 0   
\end{pmatrix}
\end{align*}
and the vector $w_{0} = (\phi^{(p-1)}(x_{0}), \phi^{(p-2)}(x_{0}), \ldots, \phi(x_{0}))^{T}$, 
the initial value problem above is expressed as 
\begin{align}
w'(z) = F(z)w(z), \quad
w(x_{0}) = w_{0}, \label{eq:trans_1nd_order_ODE}
\end{align}
where $w(z)$ corresponds to $(\psi^{(p-1)}(z), \psi^{(p-2)}(z), \ldots, \psi(z) )^{T}$. 
The equation \eqref{eq:trans_1nd_order_ODE} is equivalently expressed as
\begin{align}
w(z) = w_{0} + \int_{x_{0}}^{z} F(s)w(s) \mathrm{d}s. 
\label{eq:trans_1nd_order_ODE_equiv}
\end{align}

By the fundamental method of successive 
approximations~\cite[Theorem 2.3.1]{bib:Hille_ODE_complex}%
\footnote{
In \cite[Theorem 2.3.1]{bib:Hille_ODE_complex}, 
the general form $w' = F(z, w)$ is considered. 
Here we consider the special case that $F$ is linear with respect to $w$. 
}, 
it is shown that the equation~\eqref{eq:trans_1nd_order_ODE} 
has a unique analytic solution $w$ on $\mathcal{D}_{d_{\varepsilon}}$
for each $x_{0}\in \mathbf{R}$. 
Then what remains is to prove 
the consistency of the solutions for different $x_{0}$'s 
and the periodicity of the solution.
First, 
note that for any $x_{0}$ the solution $w$ is identical to 
$v = (\phi^{(p-1)}, \phi^{(p-2)}, \ldots, \phi)^{T}$ on $\mathbf{R}$.  
This follows from the condition
\begin{align}
\| w(x) - v(x) \| 
= \left\| \int_{x_{0}}^{x} F(s)(w(s)-v(s)) \mathrm{d}s \right\| 
\leq \sup_{z \in \mathcal{D}_{d_{\varepsilon}}} \| F(z) \| 
\int_{x_{0}}^{x} \| w(s)-v(s) \|\, |\mathrm{d}s|, 
\label{eq:unique_cond}
\end{align}
where $\| \cdot \|$ is a norm of $\mathbf{R}^{p}$. 
Then if $w_{1}$ and $w_{2}$ 
are the solutions for $x_{0}^{(1)}$ and $x_{0}^{(2)}$, respectively, 
we have
\begin{align*}
w_{1}(z) - w_{2}(z)
& = 
w_{1}(x_{0}^{(1)}) - w_{2}(x_{0}^{(2)})
+ \int_{x_{0}^{(1)}}^{z} F(s)w_{1}(s) \mathrm{d}s 
- \int_{x_{0}^{(2)}}^{z} F(s)w_{2}(s)\mathrm{d}s \\
& = 
w_{1}(x_{0}^{(1)}) - w_{2}(x_{0}^{(2)})
- \int_{x_{0}^{(2)}}^{x_{0}^{(1)}} F(s)w_{2}(s) \mathrm{d}s 
+ \int_{x_{0}^{(1)}}^{z} F(s) (w_{1}(s) - w_{2}(s)) \mathrm{d}s \\
& = 
v(x_{0}^{(1)}) - v(x_{0}^{(2)})
- \int_{x_{0}^{(2)}}^{x_{0}^{(1)}} F(s) v(s) \mathrm{d}s 
+ \int_{x_{0}^{(1)}}^{z} F(s) (w_{1}(s) - w_{2}(s)) \mathrm{d}s \\
& = 
\int_{x_{0}^{(1)}}^{z} F(s) (w_{1}(s) - w_{2}(s)) \mathrm{d}s 
\end{align*}
for any $z \in \mathcal{D}_{d_{\varepsilon}}$, 
which implies $w_{1} \equiv w_{2}$ on $\mathcal{D}_{d_{\varepsilon}}$ 
by the same criterion as \eqref{eq:unique_cond}. 
Next, the expressions
\begin{align*}
w(z + L) 
& = v(x_{0} + L) + \int_{x_{0}+L}^{z+L} F(s)w(s) \mathrm{d}s
= v(x_{0}) + \int_{x_{0}}^{z} F(t+L)w(t+L) \mathrm{d}t \\
& = v(x_{0}) + \int_{x_{0}}^{z} F(t)w(t+L) \mathrm{d}t, \\
w(z) 
& = v(x_{0}) + \int_{x_{0}}^{z} F(t)w(t) \mathrm{d}t. 
\end{align*}
guarantees the periodicity of $w$ also by the same criterion as \eqref{eq:unique_cond}. 

Finally, since the all arguments above hold true for any $\varepsilon$ with $0< \varepsilon < d$, 
we obtain the conclusion. 
\end{proof}
\renewcommand{\proofname}{Proof}

\end{document}